\documentclass{article}
\usepackage{graphicx,amsfonts,amsmath, enumerate} 
\usepackage{amssymb, amsmath, mathrsfs, amsthm}
\usepackage{graphicx}
\usepackage{color}
\usepackage{hyperref}
\usepackage[top=1in, bottom=1in, left=1.5in, right=1.5in]{geometry}
\usepackage{float, caption, subcaption}
\usepackage{lineno}
\usepackage{tikz}
\usepackage{authblk} 
\newtheorem{thm}{Theorem}[section]
\newtheorem{cor}[thm]{Corollary}
\newtheorem{prop}[thm]{Proposition}

\newtheorem{example}[thm]{Example}
\newtheorem*{defn}{Definition}
\newtheorem{obs}[thm]{Observation}

\newtheorem{question}{Question}[]

\newcommand{\hc}{\hat{c}}
\newcommand{\squarespace}{\,\square\,}

\title{Closed Neighborhood Balanced Coloring of Graphs}

\author[1]{K.~L.~Collins}
\author[2]{M.~Bowie}
\author[3]{N.~B.~Fox}
\author[4]{B.~Freyberg}
\author[5]{J.~Hook}
\author[6]{A.~M.~Marr} 
\author[7]{C.~McBee}
\author[8]{A.~Semani\v{c}ov\'{a}-Fe\v{n}ov\v{c}\'{i}kov\'a\footnote{Slovak Research and Development Agency under the contract No.~APVV-23-0191 and by VEGA 1/0243/23.}}
\author[9]{A.~Sinko}
\author[10]{A.~N.~Trenk}
\affil[1]{\small Dept. of Math. \& Comp. Sci., Wesleyan University, Middletown, CT, USA}
\affil[2]{Dept. of Math., University of North Alabama, Florence, AL, USA}
\affil[3]{Dept. of Math. \& Stats., Austin Peay State University, Clarksville, TN, USA}
\affil[4]{Combinatorica NPO, Opava, CZE}
\affil[5]{Dept. of Math. \& Comp. Sci., Mount St.\ Mary's University, Emmitsburg, MD, USA}
\affil[6]{Dept. of Math. \& Comp. Sci., Southwestern University, Georgetown, TX, USA}
\affil[7]{Dept. of Math. \& Comp. Sci., Providence College, Providence, RI, USA}
\affil[8]{Dept. of Applied Mathematics \& Informatics, Technical University of Ko\v{s}ice, SVK}
\affil[9]{Dept. of Math., Coll. of Saint Benedict \& Saint John's Univ., St. Joseph, MN, USA}
\affil[10]{Dept. of Mathematics, Wellesley College, Wellesley, MA, USA}

\date{\today}

\begin{document}

\maketitle

\begin{abstract}
A coloring of the vertex set of a graph using the colors red and blue is a closed neighborhood balanced coloring if for each vertex there are an equal number of red and blue vertices in its closed neighborhood.  A graph with such a coloring is called a CNBC graph.  Freyberg and Marr \cite{FM24} studied the related class of NBC graphs where closed neighborhood is replaced by open neighborhood.  We prove results about CNBC graphs and NBC graphs.
 We show that the class of CNBC graphs is not hereditary, that the sizes of the color classes can be arbitrarily different, and that if the sizes of the color classes are equal, then a graph is a CNBC graph if and only if its complement is an NBC graph. When the sizes of the color classes are equal, we show that the join of two CNBC graphs is a CNBC graph, and the lexicographic product of a CNBC graph with any graph is a CNBC graph.
 We prove that the Cartesian product of any CNBC graph and any NBC graph is a CNBC graph, and characterize when a hypercube is an NBC graph or a CNBC graph, but show that the product of two CNBC graphs need not be an NBC graph. We show that the strong product of a CNBC graph with any graph is a CNBC graph. We construct infinite families of circulants that are CNBC graphs, and give characterizations of CNBC trees, generalized Petersen graphs, cubic circulants and quintic circulants when $n\equiv 2 \pmod 4$.
\end{abstract}


\noindent {\small \textbf{\textit{Mathematics subject codes}} 05C15, 05C78} 


\section{Introduction}
We introduce closed neighborhood balanced colorings of graphs, which are colorings of the vertex set of a graph using the colors red and blue such that for each vertex there are an equal number of red and blue vertices in its closed neighborhood.  A graph with such a coloring is called a \emph{CNBC} graph.  Note that a closed neighborhood of a vertex $v$ in a simple graph includes vertex $v$, whereas an open neighborhood of $v$ does not include $v$ itself.  Freyberg and Marr \cite{FM24} studied the related class of NBC graphs where closed neighborhood is replaced by open neighborhood.  

Formally, let $G=(V,E)$ be a simple graph and $v \in V$. The set $N(v)=\{u:uv \in E\}$ is the open neighborhood of $v$ in $G$, and the set $N[v]=N(v) \cup \{v\}$ is the closed neighborhood of $v$ in $G$.  An $(R,B)$-coloring of $G$ is a partition of $V$ as $R \cup B$  where $R$ is the set of red vertices and $B$ is the set of blue vertices.  A \emph{closed neighborhood balanced coloring} (CNB-coloring) of $G$ is an $(R,B)$-coloring such that for each vertex $v \in V$, the number of red vertices in $N[v]$ is equal to the number of blue vertices in $N[v]$. A \emph{neighborhood balanced coloring} (NB-coloring) 
of $G$ is an $(R,B)$-coloring such that for each vertex $v \in V(G)$, the number of red vertices in $N(v)$ is equal to the number of blue vertices in $N(v)$. If $G$ admits a CNB-coloring (respectively, NB-coloring), then $G$ is a CNBC graph (respectively, NBC graph), and we say that each closed neighborhood (respectively, open neighborhood) is balanced.  

Other authors have studied related graph coloring questions. In this paper, we examine neighborhood balanced colorings of graphs, using only two colors, although the definition generalizes to any number of colors.    Neighborhood balanced graphs with three colors were studied in \cite{MS24}.  Several authors have studied proper vertex colorings of graphs with additional restrictions on neighborhoods. In \cite{MMOS24}, the authors study proper vertex colorings of graphs where for each $v\in V$, the set of colors of $N[v]$ are consecutive integers; these are called proper interval vertex colorings. In \cite{PS22}, the authors consider proper vertex colorings such that for each $v\in V$, there is a color that occurs in $N(v)$ an odd number of times.  

Both CNBC graphs and NBC graphs can model situations where a balance of perspectives or skills is desired.  For example, a graph can represent a population of people in a workplace where there is an edge between people who interact regularly.  Each person might receive training in one of two types of skills, represented by coloring that vertex red or blue.   If this coloring is balanced, then each person has access to balanced perspectives and skill sets among their close colleagues.  

Another example where it is desirable to assign vertices as red or blue in a balanced way appears in graph mining \cite{DMWCL23}.  Graph mining is the set of techniques used to analyze the properties of graphs that arise in applications, predict how their structure is related to the application, and develop models to generate graphs that exhibit many features of the graphs of interest.  Data from a variety of applications including social networks, transportation networks, and bioinformatics are well represented using graphs.  The red and blue vertices in these graphs represent the different demographic groups being considered.  Fair graph partition is a technique used in graph mining to balance the number of red and blue vertices in different contexts.  For example, in \cite{LLL24}, the authors consider the number of adjacencies between two red, two blue, or one red and one blue vertex.

In the remainder of this section we present some fundamental observations, additional notation, and an outline of the rest of the paper. 
The following observations follow immediately from our definitions. 

\begin{obs}\label{obs:odddegree}
If $G$ is a CNBC graph, then every vertex has odd degree and  therefore $G$ has an even number of vertices.
If $G$ is an NBC graph, then every vertex has even degree.
 \end{obs}

\noindent For instance, the cycle $C_n$ is never a CNBC graph by Observation~\ref{obs:odddegree}, and it is an NBC graph if and only if $n \equiv 0 \pmod  4$, see \cite{FM24}.  The complete graph $K_n$ is a CNBC graph if and only if $n$ is even, since the degree of each vertex must be odd, and we can color half of the vertices red and half blue. 

\begin{obs}\label{neighborhoods} Let $G$ be a graph and $u,v\in V(G)$. 
If $N(u)=N(v)$, then $u$ and $v$ are the same color in any CNB-coloring of $G$. 
If $N[u]=N[v]$, then $u$ and $v$ are the same color in any NB-coloring of $G$.
\end{obs}

\noindent For example, the complete bipartite graph $K_{m,n}$ with $m$-set $M$ and $n$-set $N$ is a CNBC graph if and only if $m=n=1$, since every  vertex in $M$ has the same open neighborhood, and every vertex in $N$ has the same open neighborhood.

We study CNBC and NBC graphs by analyzing graph families and several graph products, and different colorings. In addition, we find relationships between neighborhood balanced colorings and closed neighborhood balanced colorings. 
Throughout the paper, we let $G=(V,E)$ be a simple graph, with vertex set $V$ and edge set $E$. Let $\overline{G}$ be the complement of $G$. Also, for disjoint graphs $G$ and $H$, let $G+H$ be the graph with vertex set $V(G)\cup V(H)$ and edge set $E(G)\cup E(H)$.  For a set $S$, we use the notation $|S|$ for the cardinality of $S$. Our general graph theory notation follows \cite{West}.

The rest of the paper is organized as follows.  In Section~\ref{hered-etc}   we prove that every graph is an induced subgraph of an NBC graph and of a CNBC graph.   Using additions of subgraphs with four vertices, we construct the unique six vertex CNBC tree,   and show that there exist CNBC graphs in which   the sizes of the color classes may differ by an arbitrary amount.  We also show that   if the sizes of the color classes are equal, then a graph is a CNBC graph if and only if its complement is an NBC graph. The latter result yields  corollaries about the joins of CNBC graphs and the lexicographic
product of any graph with a CNBC graph. In Section~\ref{sect:count}, we provide results about the number of vertices of each color in CNBC graphs  and NBC graphs, and also about the number of edges whose endpoints have specified colors. The class of  CNBC trees is characterized in Section~\ref{trees}.
In Section~\ref{sect:regular}, we provide counting results for regular graphs, and apply them to  particular families of regular graphs including cubic graphs, circulants, and generalized Petersen graphs.   We construct infinite families of NBC and CNBC circulants, partially characterize quintic CNBC circulants, and characterize   cubic circulants and generalized Petersen graphs that are CNBC graphs. In Section~\ref{sub:cartesian-strong}, we consider Cartesian products and strong products and show that  that the $n$-dimensional hypercube is an NBC graph when $n$ is even and a CNBC graph when $n$ is odd. We prove that the strong product of a CNBC graph and any graph is a CNBC graph and that the Cartesian product of a CNBC graph with an NBC graph is a CNBC graph. We conclude in Section~\ref{openq} with some open questions and acknowledgments.
 
\section{Global properties of CNBC and NBC graphs} \label{hered-etc}

In this section, we prove  that the properties of being a CNBC graph or an NBC graph are not hereditary properties. In addition, we show that the number of red and blue vertices need not be equal in a CNBC graph or an NBC graph, and we consider complements, joins and the lexicographic product. We provide several examples. 

\subsection{The classes of CNBC and NBC graphs are not hereditary} \label{sub:not-here} 
A family $\cal F$ of graphs is \emph{hereditary} if $G \subseteq { \cal F}$ and $H$ an induced subgraph of $G$ together imply that $H  \subseteq { \cal F}$. For example, see \cite{S86}. Many familiar classes of graphs are hereditary, such as planar graphs, acyclic graphs, and bipartite graphs.  Hereditary classes can be characterized by providing a list of  forbidden induced subgraphs, for example, bipartite graphs are those with no induced odd cycle.  Indeed, a family of graphs is hereditary if and only if it has a forbidden induced subgraph characterization.  In this section we show that the classes of NBC graphs and CNBC graphs are not hereditary classes of graphs.  Indeed, there is no graph which is a forbidden induced subgraph for the class of NBC graphs or CNBC graphs.
	
	\begin{thm} \label{induced-subgraph}
	Every graph is an induced subgraph of an NBC graph and of a CNBC graph.
	
	\end{thm}
	
	\begin{proof}
	Let $G $ be a graph and  write $V(G) = \{v_1,v_2, \ldots v_n\}$.  We form the graph $H$  as follows.    Define the vertex set as $V(H) = \{x_1,x_2, \ldots x_n\} \cup \{y_1,y_2, \ldots y_n\}$ and the edge set as  $E(H) = \{x_ix_j, \ y_iy_j, \ x_iy_j \ | \ v_iv_j \in E(G)\}$.  Note that  $G$ is an induced subgraph of $H$.  In graph $H$, color each vertex $x_i$ red and each $y_i$ blue.  Then  each vertex of $H$ has an equal number of red and blue vertices in its open neighborhood, so $H$ is an NBC graph.
	
	Now, we form a CNBC graph $H'$ as follows.  Let $V(H') = V(H)$ and and $E(H') = E(H) \cup \{x_iy_i | 1 \le i \le n\}$. Color the vertices as in $H$, and note that in moving from $H$ to $H'$ we have added an edge between each red vertex $x_i$ and its blue twin $y_i$.  As before, $G$ is an induced subgraph of $H'$ and now each vertex of $H'$ has an equal number of red and blue vertices in its closed neighborhood, so $H'$ is a CNBC graph.
	\end{proof}

 \begin{cor}
	The classes of NBC graphs and CNBC graphs are not hereditary.
	\end{cor}
	\begin{proof}
	Let $H$ be a graph (such as $K_{1,4}$) that is neither an NBC graph nor a CNBC graph.  By Theorem~\ref{induced-subgraph}, there exists an NBC graph $G$ and a CNBC graph $G'$ that contain $H$ as an induced subgraph.  Hence the families of NBC graphs and CNBC graphs are not hereditary.
	\end{proof}

\subsection{Unequal numbers of red and blue vertices} \label{sub:unequ} 
 
The proof of Theorem~\ref{induced-subgraph} takes a graph and creates a CNBC graph that contains the original graph as a subgraph. The next definition gives a way to start with a CNBC graph and create a new CNBC graph that has four additional vertices, one of one color and three of the other color. This construction will also be used in Section~\ref{trees}. 

\begin{defn} \rm 
	{\rm 
	Given a CNB-coloring of a graph $G$ and $z \in V(G)$, a 
	 \emph{$4$-vertex addition at $z$} is the operation of adding vertices $v$ and $x$ adjacent to $z$; and adding vertices $w_1$ and $w_2$ adjacent to $v$; and assigning $z$'s color to $v$ and the opposite color to  $x, w_1$ and $w_2$.
	}
 \label{four-vtx-addn}
	\end{defn}

	\begin{prop} \label{prop:4vt-add}
	Given a CNB-coloring of a graph $G$, let $G'$ be the graph resulting from a $4$-vertex addition at a vertex $z$.  Then $G'$ is a CNBC graph and $G'$ has one additional vertex of $z$'s color and three additional vertices of the opposite color. 
    
  \label{four-vertex-addition-lemma}
	\end{prop}
	
	\begin{proof}
 It is straightforward to verify that for every vertex in $G'$, the number of red vertices in its closed neighborhood equals the number of blue vertices in its closed neighborhood. The $4$-vertex addition adds one vertex of $z$'s color (namely $v$) and three vertices of the opposite color ($x$, $w_1$, and $w_2$).
	\end{proof}

\begin{figure}[ht]
    \centering
        \begin{tikzpicture}[scale=.28]

        \draw[line width=5mm, gray, opacity=0.75] (13,0) -- (20,0);
            \draw[line width=3mm, gray, opacity=0.75] (-9,0) -- (-1,0);

        \draw[fill=white] (20,0) circle [radius=3.5];
        \draw[fill=white]  (10,0) circle [radius=3.5];
        \draw[fill=white]  (-1,0) circle [radius=3.5];
            \draw[fill=white] (-10,0) circle [radius=2.5];
                    
         \node[circle, draw, fill=white,  minimum size=.3cm, inner sep=0pt, label=left:$x$] (x) at (-22,2.3){};
            \node[circle, draw, fill=white,  minimum size=.3cm, inner sep=0pt, label=left:$z_1$] (z1) at (-22,0){};
            \node[circle, draw, fill=white,  minimum size=.3cm, inner sep=0pt, label=left:$z_2$] (z2) at (-22,-2.3){};
           \node[circle, draw, fill=white,  minimum size=.3cm, inner sep=0pt, label=right:$w_1$] (w1) at (-20,2.3){};
            \node[circle, draw, fill=white,  minimum size=.3cm, inner sep=0pt, label=right:$v$] (v) at (-20,0){};
            \node[circle, draw, fill=white,  minimum size=.3cm, inner sep=0pt, label=right:$w_2$] (w2) at (-20,-2.3){};
            
            \draw (x) -- (z1);
            \draw (z1) -- (z2);
            \draw (z1) -- (v);
            \draw (w1) -- (v);
            \draw (v) -- (w2);

            \node[] at (-20.8,-5){$H_6$};
        
            \node[circle, draw, fill=white,  minimum size=.3cm, inner sep=0pt] (x) at (9,2){};
            \node[circle, draw, fill=white,  minimum size=.3cm, inner sep=0pt] (z1) at (9,0){};
            \node[circle, draw, fill=white,  minimum size=.3cm, inner sep=0pt] (z2) at (9,-2){};
           \node[circle, draw, fill=white,  minimum size=.3cm, inner sep=0pt] (w1) at (11,2){};
            \node[circle, draw, fill=white,  minimum size=.3cm, inner sep=0pt] (v) at (11,0){};
            \node[circle, draw, fill=white,  minimum size=.3cm, inner sep=0pt] (w2) at (11,-2){};
            
            \draw (x) -- (z1);
            \draw (z1) -- (z2);
            \draw (z1) -- (v);
            \draw (w1) -- (v);
            \draw (v) -- (w2);
        
            \node[circle, draw, fill=white,  minimum size=.3cm, inner sep=0pt] (x') at (19,2){};
            \node[circle, draw, fill=white,  minimum size=.3cm, inner sep=0pt] (z1') at (19,0){};
            \node[circle, draw, fill=white,  minimum size=.3cm, inner sep=0pt] (z2') at (19,-2){};
           \node[circle, draw, fill=white,  minimum size=.3cm, inner sep=0pt] (w1') at (21,2){};
            \node[circle, draw, fill=white,  minimum size=.3cm, inner sep=0pt] (v') at (21,0){};
            \node[circle, draw, fill=white,  minimum size=.3cm, inner sep=0pt] (w2') at (21,-2){};
            
            \draw (x') -- (z1');
            \draw (z1') -- (z2');
            \draw (z1') -- (v');
            \draw (w1') -- (v');
            \draw (v') -- (w2');

            \node[] at (15.2,-5){$K_2 \circ H_6=H_6 \vee H_6$};
 
            \node[circle, draw, fill=white,  minimum size=.3cm, inner sep=0pt] (x) at (-10,1){};

            \node[circle, draw, fill=white,  minimum size=.3cm, inner sep=0pt] (z2) at (-10,-1){};
                 
            \draw (x) -- (z2);
               
            \node[circle, draw, fill=white,  minimum size=.3cm, inner sep=0pt] (x') at (-2,2){};
            \node[circle, draw, fill=white,  minimum size=.3cm, inner sep=0pt] (z1') at (-2,0){};
            \node[circle, draw, fill=white,  minimum size=.3cm, inner sep=0pt] (z2') at (-2,-2){};
           \node[circle, draw, fill=white,  minimum size=.3cm, inner sep=0pt] (w1') at (0,2){};
            \node[circle, draw, fill=white,  minimum size=.3cm, inner sep=0pt] (v') at (0,0){};
            \node[circle, draw, fill=white,  minimum size=.3cm, inner sep=0pt] (w2') at (0,-2){};
            
            \draw (x') -- (z1');
            \draw (z1') -- (z2');
            \draw (z1') -- (v');
            \draw (w1') -- (v');
            \draw (v') -- (w2');

            \node[] at (-5.8,-5){$K_2 \vee H_6$};
        \end{tikzpicture}

        \caption{The graphs $H_6$ and $K_2\vee H_6$,  and also the lexicographic product $K_2 \circ H_6$ which is also the join $H_6\vee H_6$.}
        \label{fig-H6}
\end{figure}
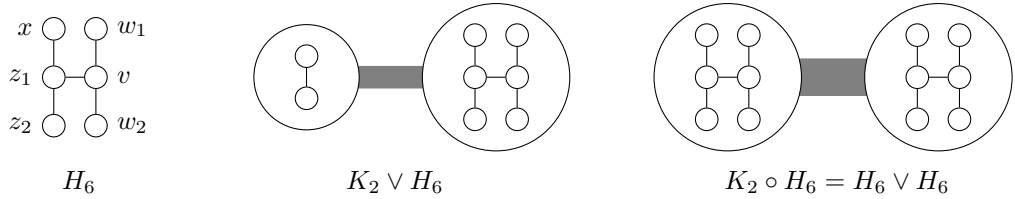

 \begin{example} \label{h6-tree} \rm In this example we construct a 6-vertex tree  $H_6$ that is a 4-vertex addition to $K_2$, and is a CNBC graph (see Figure~\ref{fig-H6}). Let the vertices of $K_2$ be $z_1$ and $z_2$, and make a 4-vertex addition at $z_1$ with $v,x,w_1, w_2$.  Since $K_2$ is a CNBC graph, by Proposition~\ref{prop:4vt-add}, $H_6$ is a CNBC graph. It is straightforward to check that in any CNB-coloring of $H_6$, the vertices $z_2, x, w_1, w_2$ are one color and $z_1, v$ are the opposite color, so that in any CNB-coloring, $H_6$ has an unequal number of red and blue vertices. We prove in Section~\ref{trees} that every 6-vertex CNBC tree is isomorphic to $H_6$. 
\end{example}
	
	\begin{cor}
	 There exist  CNBC graphs with a CNB-coloring that have arbitrarily more red vertices than blue vertices. Moreover, every CNBC graph is an induced subgraph of such a graph.
 \label{arb-more-red-cor}
	\end{cor}
	
	\begin{proof}
	Let $G$ be a CNBC graph and fix a CNB-coloring of $G$.  Repeatedly make a $4$-vertex additions at  blue vertices.    The resulting graphs all contain $G$ as an induced subgraph and by Proposition~\ref{four-vertex-addition-lemma} they are all CNBC graphs.  After each addition there are three new red vertices and one new blue vertex, so the number of red vertices is increasing faster than the number of blue vertices.
	\end{proof}

Similarly, there exist NBC graphs that have more red than blue vertices.

\begin{defn} \rm \label{three-vtx-addn}
	Given an NB-coloring of a graph $G$ where $w,x$ are one color and $y,z$ are the opposite color, a
	 \emph{$3$-vertex addition at $\{w,x,y,z\}$} 
  is the operation of adding vertices $u, a_1, a_2$ such that $u$ is  adjacent to all of $w,x,y,z$; and $a_1$ is adjacent to $w$ and $y$; and $a_2$ is adjacent to $x $ and $z$; and assigning one color  to $u$ and the opposite color to $a_1$ and  $ a_2$. 
	\end{defn}

 \begin{prop} \label{prop-3vt-add}
	Given an NB-coloring of a graph $G$, if  $G'$ is the graph resulting from a $3$-vertex addition at a set $\{w,x,y,z\}$, then $G'$ is an NBC graph.   Moreover, $G'$ has one additional vertex of one color and  two additional vertices of the opposite color.
	\label{three-vertex-addition-prop}
	\end{prop}
 \begin{proof}
 It is straightforward to verify that for every vertex in $G'$, the number of red vertices in its open neighborhood equals the number of blue vertices in its open neighborhood. The $3$-vertex addition adds one  additional vertex of  one color  ($u$)
 and two vertices of the opposite color ($a_1, a_2$).
 \end{proof}

\begin{cor}
    There exist  NBC graphs that have an  NB-coloring with arbitrarily more red vertices than blue vertices. Moreover, every NBC graph is the induced subgraph of such a graph.
\end{cor}

\begin{proof} 

Let $G$ be an NBC graph and fix an NB-coloring of $G$.  Repeatedly make a $3$-vertex additions to $G$, and color the new  degree 4 vertex blue each time. The resulting graphs all contain $G$ as an induced subgraph and by Proposition~\ref{three-vertex-addition-prop}, they are all NBC graphs.  After each addition there are two new red vertices and one new blue vertex, so the number of red vertices is increasing faster than the number of blue vertices.
\end{proof}

\subsection{Complements, joins, and lexicographic products} \label{sub:compl}

In the previous section we studied CNBC and NBC graphs that had colorings in which the number of red and blue vertices are unequal.  In contrast,  we consider colorings with an equal number of red and blue vertices.  We write $N_H(v)$ for the open neighborhood of $v$ in graph $H$ and $N_H[v]$ for the closed neighborhood. 

\begin{thm}\label{thm:complement}  
 A coloring $(R,B)$  that has  $|R| = |B|$ is an NB-coloring of  graph $G$ if and only if it is a CNB-coloring of $\overline{G}$. 
\end{thm}

\begin{proof}
Let $(R,B)$ be an NB-coloring of $G$ with $|R| = |B| = k$. Take any vertex $v\in V(G)$. By definition, $|N_{G}(v)\cap R|=j=|N_{G}(v)\cap B|$ for some $j$.  If $v$ is blue, then $v$ has $k-j$ red neighbors  and $k-j-1$ blue neighbors in $\overline{G}$, thus $|N_{G}[v]\cap R|=k-j=|N_{G}[v]\cap B|$. A similar argument holds for red vertices, thus  $(R,B)$ is a CNB-coloring for $\overline{G}$.

Conversely, Let $(R,B)$ be a CNB-coloring of $\overline{G}$ with $|R| = |B| = k$. Take any vertex $v\in V(\overline{G})$. By definition, $|N_{\overline{G}}[v]\cap R|=j=|N_{\overline{G}}[v]\cap B|$ for some $j$. If $v$ is blue then $v$ has $j-1$ blue neighbors in $\overline{G}$ and thus $|N_{G}(v)\cap B| = k-j$, and $v$ has $j$ red neighbors in $\overline{G}$ and thus $|N_{G}(v)\cap R| = k-j$.  A similar argument holds for red vertices, thus  $(R,B)$ is an NB-coloring for $G$.
\end{proof}

For example, the smallest NBC graph with $|R|=|B|$ is $\overline{K_2}$, and, since every vertex has odd degree in a CNBC graph, $K_2$ is the smallest CNBC graph. The next example shows that the hypothesis that $|R|=|B|$ is necessary. 

\begin{example} \label{H7} \rm 
In this example, we construct an 7-vertex graph $H_7$ that is an NBC graph and whose complement is not a CNBC graph.  Start with $\overline{K_4}$ and the NB-coloring that has two red and two blue vertices.  Make a $3$-vertex addition at these four vertices.   By Propositon~\ref{prop-3vt-add},  the resulting graph $H_7$ is an NBC graph, and has an NB-coloring with 4 vertices of one color and 3 of the opposite color. Each vertex of $H_7$ has even degree and $7$ is odd, so each vertrex in  $\overline{H_7}$  also has even degree and $\overline{H_7}$ is not a CNBC graph.
\end{example}

\bigskip

Recall that $G\vee H$ is the join of graphs $G$ and $H$, which by definition is $G+H$ with additional edges $\{gh\;|\; g\in V(G), h\in V(H)\}$. Figure~\ref{fig-H6} shows the join $K_2 \vee H_6$.
\begin{prop} \label{prop:join}

Let $G_1$ have a coloring $(R_1, B_1)$ and $G_2$ have a coloring $(R_2, B_2)$ such that $|R_1|=|B_1|$ and $|R_2|=|B_2|$.  If both $(R_1,B_1)$ and $(R_2, B_2)$ are CNB-colorings (respectively, NB-colorings) then  $G_1 \vee G_2$ is a CNBC graph (respectively, an NBC graph).
\end{prop} 

\begin{proof} If $(R_1, B_1)$ and $(R_2, B_2)$ are CNB-colorings (respectively, NB-colorings), then for each vertex in $G_1\vee G_2$, the number of additional red and blue neighbors are equal, since $|R_1|=|B_1|$ and $|R_2|=|B_2|$.
\end{proof}

Proposition~\ref{prop:join} may also be proved using Theorem~\ref{thm:complement}, since $\overline{G_1\vee G_2}=\overline{G_1}+\overline{G_2}$. 
The hypothesis of $|R|=|B|$ is necessary and in the following we provide a counterexample if the hypothesis is dropped.

 \begin{example} \rm  In this example we show that there exist two CNBC graphs whose join is not a CNBC graph.  Let $H_6$ be the graph from Example~\ref{h6-tree} with $V(H_6) = \{v,x,z_1,z_2,w_1,w_2\}$ and consider the join of  $K_2$ and $H_6$ (see Figure~\ref{fig-H6}). Both $K_2$ and $H_6$ are CNBC graphs.  For a contradiction, assume that $K_2\vee H_6$ is a CNBC graph and let $(R,B)$ be a CNB-coloring of it. The closed neighborhood  in $K_2\vee H_6$ of a vertex in $K_2$ is $V(K_2\vee H_6)$, and this neighborhood must be balanced, so $|R| = |B|$.  Now $\overline{K_2 \vee H_6}=\overline{K_2}+\overline{H_6}$, and by Theorem~\ref{thm:complement}, $(R,B)$ is an NB-coloring of this graph.     The only edges in $\overline{K_2}+\overline{H_6}$ are in $\overline{H_6}$, so $(R,B)$ induces an NB-coloring of $\overline{H_6}$. By Observation~\ref{neighborhoods}, the color of $w_1$ equals the color of $w_2$.  However, 
 $N_{\overline{H_6}}(z_1)=\{w_1, w_2\}$, so this neighborhood is not balanced, a contradiction since $(R,B)$ induces an NB-coloring of $\overline{H_6}$.
 \end{example}

We recall the lexigraphic product of graphs $G$ and $H$. It is formed by putting a copy of $H$ in for each vertex of $G$, and then for every pair of adjacent vertices in $G$ adding a complete bipartite graph between their copies of $H$. Formally, the \textit{lexicographic product}  $G \circ H$ of graphs $G$ and $H$ has vertex set $V(G) \times V(H)$ and vertices $(g,h)$ and $(g', h')$ are adjacent in $G \circ H$ if and only if either  $gg' \in E(G)$ or  $g=g'$ and $hh' \in E(H)$. 

 \begin{cor}\label{thm:lexi}
    Let $G$ be any graph and suppose that $H$ be admits a CNB-coloring $(R,B)$ such that $|B|=|R|$.  Then $G \circ H$ is a CNBC graph.
\end{cor}

\begin{proof} 
Theorem~6 in \cite{FM24} says that for any graphs $W_1, W_2$ where $W_2$ is an NBC graph with an NB-coloring $(R,B)$ such that  $|R|=|B|$, then $W_1\circ W_2$ is an NBC graph. In the proof, each copy of $W_2$ is colored the same. Thus, for $W_1\circ W_2$, $|R|=|B|$.  Let $W_1=\overline{G}$ and $W_2=\overline{H}$. By Theorem~\ref{thm:complement}, $\overline{H}$ is an NBC graph with the same coloring as $H$. Thus, $\overline{G}\circ \overline{H}$ is an NBC graph, with a coloring where $|R|=|B|$. Now $\overline{\overline{G}\circ \overline{H}} = G\circ H$. By Theorem~\ref{thm:complement}, $G\circ H$ is a CNBC graph.  
\end{proof}

The following example shows that the hypothesis of $|B|=|R|$  in Corollary~\ref{thm:lexi} is necessary. 

\begin{example} \rm 
In this example we show that there exist a graph $G$ and a CNBC graph $H$ such that $G\circ H$ is not a CNBC graph. Let $G=K_2$, and $H=H_6$ from Example~\ref{h6-tree}. Then $K_2\circ H_6$ is the join of two copies of $H_6$. Let the vertices of one copy be $z_1, z_2, x,v, w_1, w_2$ and the vertices of the other be $z'_1, z'_2, x',v', w'_1, w'_2$. Assume that $(R,B)$ is a CNB-coloring of $K_2\circ H_6$. Since $N_{K_2\circ H_6}(w_1)=N_{K_2\circ H_6}(w_2)$, by Observation~\ref{neighborhoods}, the color of $w_1$ is the same as the color of $w_2$. Now  $N_{K_2\circ H_6}[z_1]=V(K_2\circ H_6)-\{w_1, w_2\}$, and is balanced, so has five red and five blue vertices. Without loss of generality, let $w_1, w_2$ be red. Then $|R|=7$ and $|B|=5$. By similar reasoning, $w'_1$ and $w'_2$ are the same color, and $N_{K_2\circ H_6}[z_1']$ has 5 red and 5 blue vertices, so $w'_1$ and $w'_2$ are red. By the reflective symmetry of $H_6$, it follows that $z_2, x, z'_2, x'$ are also red, but then there are at least 8 red vertices. This contradiction shows $K_2\circ H_6$ is not a CNBC graph.
    
\end{example}
	
\section{Counting results}\label{sect:count}

In this section, we present several counting results that will be useful for characterizing CNBC and NBC graphs. 
	
	\begin{thm} Let $G$ be a graph and $(R,B)$ be a CNB-coloring of $G$. Then  $|R| + \sum_{v \in R} \deg(v)$ $ =$ $ |B| + \sum_{v \in B} \deg(v) $.
	 \label{degree-thm}
	\end{thm}

	\begin{proof}
		Consider the incidence matrix $M$ in which the rows are indexed by $V(G)$, the columns by $R$, and for $v_i \in V(G)$, $v_j \in R$ we have $M_{v_iv_j} = 1$ if $v_j \in N[v_i]$ and otherwise, $M_{v_iv_j} = 0$.    There are $\deg(v_j) + 1$ ones in the column corresponding to the red vertex $v_j$, so summing over all columns we calculate the total  number of ones in $M$ is $\sum_{v \in R} (\deg(v) + 1) = |R| + \sum_{v \in R} \deg(v) $. Analogously, define the incidence matrix $M'$ in which the  columns  are indexed by $B$, and for $v_i \in V(G)$, $v_j \in B$ we have $M'_{v_iv_j} = 1$ if $v_j \in N[v_i]$ and otherwise, $M'_{v_iv_j} = 0$. The total  number of ones in $M'$ is $\sum_{v \in B} (\deg(v) + 1) = |B| + \sum_{v \in B} \deg(v) $. 
		
		\smallskip
		
		Now $N[v_i]$ has an equal number of red and blue vertices because $(R,B)$ is CNB-coloring.  Thus, for each $i$, row $v_i$ of $M$ has the same number of ones as row $v_i$ of $M'$, and consequently, matrices $M$ and $M'$ have the same total number of ones.  Therefore, 	$|R| + \sum_{v \in R} \deg(v) = |B| + \sum_{v \in B} \deg(v) $. 
\smallskip
	\end{proof}

  \begin{example} \rm  In this example we show that the wheel $W_n$ is a CNBC graph if and only if $n=3$. The wheel $W_n$ consists of the cycle $C_n$ and universal vertex $u$ adjacent to each vertex on the cycle. Thus, $u$ has degree $n$ and the other vertices have degree 3. Suppose that $(R,B)$ is a CNB-coloring of $W_n$. Now $N[u]=V(W_n)$, and hence $|R|=|B|=\frac{n+1}{2}$. Without loss of generality, let $u$ be blue, then the cycle contains $\frac{n-1}{2}$ blue vertices and $\frac{n+1}{2}$ red vertices. Then $\sum_{v \in B} \deg(v)=n+\frac{3(n-1)}{2}$, and  $\sum_{v \in R} \deg(v) =\frac{3(n+1)}{2}$. By Theorem~\ref{degree-thm}, these are equal. Solving for $n$, we get $n=3$. 

\end{example}

For additional counting results, we consider edges based on the colors of their endpoints.

\begin{defn} \rm  \label{defn: color number} Let $G$ be a graph and $(R,B)$ be a CNB-coloring of $G$.
We let  $|BB|$ be the number of edges with two blue endpoints,  $|RR|$   be the number of edges with two red endpoints,  and  $|BR|$  be the number of edges with one red and one blue endpoint.
\end{defn}

\begin{cor} For any CNB-coloring $(R,B)$ of a graph $G$ if $|R|=|B|$, then $|RR|=|BB|$. \label{redequblue}
\end{cor} 

\begin{proof}
By Theorem~\ref{degree-thm}, $|R|+\sum_{v\in R} \deg(v) =
|B|+ \sum_{v\in B} \deg(v)$. Now, $\sum_{v\in R} \deg(v) = 2|RR|+|RB|$, and 
$\sum_{v\in B} \deg(v) = 2|BB|+|RB|$, thus
$|R|+ 2|RR| + |RB| = |B|+ 2|BB|+|RB|$.  By hypothesis we have $|R|=|B|$, so  we conclude $|RR|=|BB|$.
\end{proof}

The next corollary is an analogue of Theorem~\ref{degree-thm} and Corollary~\ref{redequblue} for NBC graphs. 
Note that the conclusion $|RR|=|BB|$ is proven as Theorem~1 in \cite{FM24}.

\begin{cor} \label{nb:redequblue}
If $(R,B)$ is an  NB-coloring  of a graph $G$, then  $\sum_{v \in R} \deg(v) =\sum_{v \in B} \deg(v) $, and $|RR|=|BB|$. 
\end{cor}

\begin{proof}
Use the same incidence matrices as in the proof of Theorem~\ref{degree-thm}, and a similar argument to the proof of Corollary~\ref{redequblue}.
\end{proof}

We next provide another counting result that applies to all CNBC graphs. 
 	\begin{thm}
Let $G$ be a CNBC graph.  If $|E(G)| $ is even, then $|V(G)| \equiv 0 \pmod  4$, and if $|E(G)| $ is odd, then $|V(G)| \equiv 2 \pmod 4$.
		\label{parity-thm}
	\end{thm}
	
	\begin{proof}
	Let $G$ be a CNBC graph and  let $(R,B)$ be a CNB-coloring of $G$.  Thus, $|V(G)| = |R| + |B|$. By Observation~\ref{obs:odddegree}, $G$ has an even number of vertices, so $|R|$ and $|B|$ have the same parity.  
	Let $x = \sum_{v \in R} \deg(v)$ and let $y = \sum_{v \in B} \deg(v)$, so by Theorem~\ref{degree-thm}, we know $|R| + x = |B| + y$.  Solving for $y$ yields $y = |R| - |B| + x$.  By Observation~\ref{obs:odddegree},  $\deg(v)$ is odd for each vertex $v$ in $G$, so $x$ and $|R|$ have the same parity.  We observed earlier that $|R|$ and $|B|$ have the same parity, so $x$ and $|B|$ also have the same parity.
  
  Using the degree sum formula for graphs and substituting for $y$ from above we get
  $2|E(G)| = \sum_{v \in V(G)} \deg(v) = x + y = x + (|R| - |B| + x)$.   Thus $|E(G)| = x + \frac{1}{2} (|R| - |B|) = x +  \frac{1}{2} (|R| + |B|)  - |B|$.  Since $|V(G)| = |R| + |B|$ we have $\frac{1}{2} |V(G)| = |E(G)| - x  + |B|$.  As shown above, $x$ and $|B|$ have the same parity, thus $|E(G)|$ and $\frac{1}{2}|V(G)|$ have the same parity.  If $|E(G)|$ is even, then $\frac{1}{2}|V(G)|$ is even and $|V(G)| \equiv 0 \pmod  4$.  
 If $|E(G)|$ is odd then $\frac{1}{2}|V(G)|$ is odd, but $|V(G)|$ is even, so  and $|V(G)| \equiv 2 \pmod  4$.	
	\end{proof}

Further information about the numbers of red and blue vertices is provided in the next theorem. 

 \begin{thm} \label{lem:degrees 3(4) and 1(4)}
 Let $(R,B)$ be a CNB-coloring of graph $G$. There are an even number of red vertices whose degree is congruent to 3 modulo 4 and also an even number of blue vertices whose degree is congruent to 3 modulo 4.  In addition, there are an even number of vertices in $G$ whose degree is congruent to 1 modulo 4. 
 \end{thm}

\begin{proof} Let $G$ be a graph and  $(R,B)$ be a CNB-coloring of $G$. For each blue vertex $w$, let $\deg_B(w)$ be its number of blue neighbors and $\deg_R(w)$ be its number of red neighbors. Then $\deg(w)=
\deg_B(w)+\deg_R(w)$, and, since $(R,B)$ is a CNB-coloring and the color of $w$ is blue, $\deg_R(w)=1+\deg_B(w)$. Hence $\deg(w) = 2\deg_B(w)+1$. Counting only  edges in $G$ with two blue endpoints, we have that $2|BB|=\sum_{w\in B} \deg_B(w)= \sum_{w\in B}(\deg(w)-1)/2$. Thus, $4|BB|=\sum_{w\in B}(\deg(w)-1)=(\sum_{w\in B}\deg(w)) -|B|$, and we have
$4|BB|+|B|=\sum_{w\in B}\deg(w)$. Thus,  $|B|\equiv \sum_{w\in B}\deg(w) \pmod 4$.

Since $G$ is a CNBC graph, every vertex has odd degree. Let $j$ be the number of blue vertices whose degree is congruent to 1 modulo 4 and $k$ be the number of blue vertices whose degree is congruent to 3 modulo 4. Then $|B|=j+k$, and we have $|B|\equiv \sum_{w\in B}\deg(w)\equiv j+3k \pmod 4$. Hence $j+k\equiv j+3k\pmod 4$, and thus $0\equiv 2k \pmod 4$. Hence, $k$ is even. Similarly, we get that the number of red vertices whose degree is congruent to 3 modulo 4 is also even. 
By Observation~\ref{obs:odddegree}, 
$|V(G)|$ is even and every vertex has odd degree, so the number of vertices of $G$ whose degree is congruent to 1 modulo 4 is also even.
\end{proof}

\section {Trees} \label{trees}

In this section, we characterize CNBC trees. We begin with some results about the number of leaves incident to a vertex and the number of vertices in an CNBC tree.

In any CNB-coloring of a graph, each leaf must have the opposite color from that of its unique neighbor. Thus we have the following observation.

\begin{obs} \label{obs:pendants}
  Every vertex in a  CNBC graph is adjacent to at most   $(\deg(v)+1)/2$ leaves.
\end{obs} 

\noindent For example, $H_6$ in Example~\ref{h6-tree} has two vertices of degree 3, each adjacent to two leaves, and $\frac{3+1}{2}=2$. 

Using our counting results, we prove the following theorem about the number of vertices in any  CNBC tree. 

\begin{thm}
If $T$ is a CNBC tree, then $|V(T)|\equiv 2\pmod 4$.
\label{num-vert-in-tree-thm}
\end{thm}

\begin{proof}
	Let $T$ be a tree that is a CNBC graph. By Observation~\ref{obs:odddegree}, tree $T$ has an even number of vertices, and since $T$ is a tree we also know $|E(T)| = |V(T)| - 1$.  Thus $T$ has an odd number of edges and the result follows from Theorem~\ref{parity-thm}.
	\end{proof}

Recall the definition of a 4-vertex addition at $z$ from Section~\ref{sub:unequ}.
The following theorem gives a constructive characterization of which trees are CNBC graphs.  
	
	\begin{thm}
	A tree is a CNBC graph if and only if it can be constructed from $K_2$ by a sequence of $4$-vertex additions.
 \label{tree-charac-thm}
	\end{thm}
	
	\begin{proof}
	Suppose $T$ is constructed from $K_2$ by a sequence of $4$-vertex additions. The operation of a $4$-vertex addition maintains connectivity and does not create cycles, so $T$ is a tree.  We know that $K_2$ is  CNBC graph and Proposition~\ref{four-vertex-addition-lemma} implies that $T$ is also a CNBC graph.
	
	For the converse, suppose that $T$ is a CNBC tree and fix a CNB-coloring $(R,B)$ of $T$.    We know that $|V(T)|  \equiv 2  \pmod 4$  by Theorem~\ref{num-vert-in-tree-thm},  so we can let $|V(T)| = 4k+2$  where $k$ is an integer.    If $k = 0$ then $T = K_2$ and the result follows (with no 	$4$-vertex additions needed).  Otherwise, we may assume $k \ge 1$ so $|V(T)| \ge 6$.

Consider a    path $P=v_1, v_2, v_3, \ldots, v_{\ell}$ of maximum length in $T$, thus $v_1$ and $v_{\ell}$ are leaves.  Note that $\ell \ge 4$ since the star $K_{1,m}$ is not a CNBC graph for $m \ge 3$. By the maximality of $P$, all neighbors of $v_2$ (other than $v_3$) are leaves.  We know $\deg(v_2)$ is odd by Observation~\ref{obs:odddegree}, so $v_2$ has a leaf neighbor $u$ with $u \neq v_1$. 
 By Observation~\ref{obs:pendants}, $v_2$ is not adjacent to any other leaves, so $\deg(v_2) = 3$. Without loss of generality, we may assume $v_1$ is red, thus $v_2$ is blue and $u$ is red. Since $\deg(v_2) = 3$, and two of its neighbors are red, the third neighbor, $v_3$ must be blue.

We next claim that $v_3$ is adjacent to a leaf.  For a contradiction, assume $v_3$ has no leaf neighbor.  It must have at least one neighbor in addition to $v_2$ and $v_4$ since $\deg(v_3)$ is odd.  Let $y$ be such a neighbor and since $y$ is not a leaf, it has a neighbor $y' \neq v_3$.  Vertex $y'$ must be a leaf by the maximality of $P$.  However $\deg(y)$ is odd, so $y$ has another leaf neighbor $y''$.  It cannot have a third leaf neighbor by  Observation~\ref{obs:pendants}.  Therefore, vertex $y$ has degree 3 and  is adjacent to the blue vertex $v_3$ and two leaves $y'$ and $y''$, so $y$ must be blue and $y'$ and $y''$ must be red.  By the same reasoning, any non-leaf neighbor of $v_3$ (other than $v_4$) must be blue.  Thus $N[v_3]$ contains at least three blue vertices ($v_2, v_3, y$) and at most one red vertex ($v_4$).  This contradicts our assumption that our coloring is a CNB-coloring.  Thus $v_3$ has a leaf neighbor $x$ and $x$ is red.

Let $T_1$ be the tree resulting from removing the four vertices $v_1, v_2, u,x$ and their incident edges from $T$.  Thus $|V(T_1)| = |V(T)| - 4 = 4k-2$.  The colors remaining on the vertices of $T_1$ give a CNB-coloring of $T_1$ since we have removed one blue neighbor and one red neighbor from $v_3$, and no incident vertices from any other vertex of $T_1$.      

Tree $T$ is obtained from $T_1$ by a $4$-vertex addition at $v_3$.  
Apply the above argument $k$ times to get a sequence of trees $T$, $T_1$, $T_2$, \ldots, $T_k$ where $T_i$ is obtained from $T_{i+1}$ by a $4$-vertex addition and $|V(T_{i+1})| = |V(T_i)| - 4 $.   Thus $|V(T_k)| = |V(T)| - 4k = 2$, so $T_k = K_2$ and $T$ can be constructed from $K_2$ by a sequence of $4$-vertex additions.
\end{proof}

We present several corollaries of Theorem~\ref{tree-charac-thm}. The first two are direct consequences of the proof of Theorem~\ref{tree-charac-thm}.

\begin{cor}
    Every CNBC tree is of the form $T_n$ where $T_1=K_2$ and for $n>1$, $T_n$ is obtained from a tree $T_{n-1}$ by attaching to a vertex $v\in V(T_{n-1})$ an edge to a leaf and an edge to the central vertex of a path $P_3$.
\end{cor}

Up to isomorphism, there is only one way to make a 4-vertex addition to $K_2$. Thus, graph $H_6$ in Figure~\ref{fig-H6} is the only CNBC tree with 6 vertices, up to isomorphism. 

\begin{cor}  \label{arb-more-red-tree}
    For any integer $k$, there exists a CNBC tree for which  $|R|-|B| = 2k$. 
\end{cor}

\begin{proof}
  Starting with $K_2$ and a CNB-coloring of $K_2$, make $k$ new 4-vertex-additions to a blue vertex. By Proposition~\ref{prop:4vt-add}, and the proof of Theorem~\ref{tree-charac-thm}, the resulting graph will have $2k$ more red than blue vertices. 
\end{proof}

Note that $H_6$ has either 4 red and 2 blue vertices or 4 blue and 2 red vertices, compared to $K_2$ with 1 red and 1 blue. For any graph $G$ we write $\Delta(G)$ for the maximum degree of a vertex in $G$.

\begin{cor} \label{tree-max-deg} 
If $T$ is a CNBC tree and $|V(T)|=n$, then $\Delta(T)\leq \frac{n}{2}$. 
\end{cor}

\begin{proof}
Since $T$ is a CNBC tree, by Theorem~\ref{num-vert-in-tree-thm}, then $n=2+4m$ for some integer $m$. 
We prove the corollary by induction on the number of 4-vertex-additions made to $K_2$. The degree of each vertex in $K_2$ is 1, which is half of the order of $K_2$. This shows that the base case of the induction holds. 

 Now suppose that any CNBC tree $T$ that is created by $m$ 4-vertex-additions to $K_2$ satisfies $\Delta(T)\leq 1+2m$. When we make one additional 4-vertex-addition to such a tree $T$, then the result has $2+4(m+1)=6+4m$ vertices. By the proof of Theorem~\ref{tree-charac-thm}, exactly one vertex $z$ of $T$ has a larger degree in the new tree, and its degree increases by 2. Half of $6+4m$ is $3+2m=2+(1+2m)$. Thus, the degree of $z$ is less than or equal to half the number of vertices in the new tree. Each new vertex added in the 4-vertex addition has degree less than or equal to 3, which is less than half of the number of vertices in the new tree, which has at least 6 vertices. 
\end{proof}

\begin{cor} 
Let $T$ be a CNBC tree of order $n$. Then the diameter of $T$ satisfies $\mbox{\rm diam}(T)\leq \frac{n}{2}$. 
\end{cor}

\begin{proof} 
Since $T$ is a CNBC tree, we know  $n=2+4m$ by Theorem~\ref{num-vert-in-tree-thm}.
In a 4-vertex-addition at a vertex $z$, the paths from any vertex of $T$ to any of the $4$ new vertices contains $z$ and the distance from $z$ to any of the new vertices is $1$ or $2$.  Thus, each 4-vertex-addition increases the diameter of the resulting tree by at most 2. Since the diameter of $K_2$ is 1, by induction, a tree that is formed by $m$ 4-vertex-additions has diameter at most $1+2m$. 
\end{proof} 
\section{Regular graphs} \label{sect:regular} 

In this section, we consider graphs with CNB-colorings for which the number of red vertices equals the number of blue vertices, including regular graphs and some special families of graphs that are regular.

\subsection{Counting results for regular graphs}\label{sub:reg-count}

We provide results about regular graphs in this section. 

	\begin{thm}
 \label{thm:blue equals red} 
 If $(R,B)$ is a CNB-coloring of an $r$-regular graph $G$, then $|R| = |B|$, $|RB| = \frac{(r+1)|V(G)|}{4}$, and $|RR| = |BB| = \frac{(r-1)|V(G)|}{8}$.

	\end{thm}
	
\begin{proof}
Each red vertex has $\frac{r-1}{2}$ red neighbors and $\frac{r+1}{2}$ blue neighbors. First we count the edges with one red endpoint and one blue endpoint at their red endpoints, and get $|RB | =  \sum_{v \in R} \frac{r+1}{2}  =  |R| (\frac{r+1}{2}).$ When we count the same quantity at blue endpoints we get $|RB|=|B| (\frac{r+1}{2})$.  Thus $|R| = |B|= \frac{|V(G)|}{2}$  and  $|RB| =  \frac{(r+1)|V(G)|}{4}.$  

To count the edges with both endpoints red, consider the subgraph $H$ induced by the red vertices.  We have $|E(H)| = \frac{1}{2} \sum_{v \in R} \deg_H(v) = \frac{1}{2}  |R| (\frac{r-1}{2}) = \frac{1}{2} \left(\frac{|V(G)|}{2} \right)\left(\frac{r-1}{2}\right)  = \frac{(r-1)|V(G)|}{8}$.  Thus $|RR| = \frac{(r-1)|V(G)|}{8}$.
	 By symmetry we get the same result for $|BB|$.  
\end{proof}

 \begin{thm}
    \label{thm:blue equals red-nb} 
    If $(R,B)$ is an NB-coloring of an $r$-regular graph $G$, then $|R| = |B|$, $|RB| = \frac{r|V(G)|}{4}$, and $|RR| = |BB| = \frac{r|V(G)|}{8}$. 

\end{thm}

\begin{proof}
The proof is similar to the proof of Theorem~\ref{thm:blue equals red}.
\end{proof}

\begin{cor}\label{cor:singly even}
If $G$ is an $r$-regular CNBC graph, then either $|V| \equiv 0 \pmod 4$ or $r \equiv 1 \pmod 4.$
\end{cor}

\begin{proof}
By Theorem~\ref{thm:blue equals red}, for any CNB-coloring $(R,B)$ of $G$, we have $|RR|= \frac{(r-1)|V(G)|}{8}$, which therefore must be an integer.  Thus at least one of $|V|$ and $(r-1) $ must be a multiple of $4$.  \end{proof}

Several graph families that we consider are 3-regular. We record the counting facts for cubic CNBC graphs in the next corollary.

 \begin{cor} \label{3-reg-cor} Let $G$ be a CNBC graph that is regular of degree $3$. 
    Then $|V(G)| \equiv 0 \pmod 4$, and in any CNB-coloring $(R,B)$ of $G$, $|R|=|B|=\frac{|V(G)|}{2}$, and $|RB|=|V(G)|$, $|RR|=|BB|=\frac{|V(G)|}{4}$, and the monochromatic edges form a  matching of all vertices of $G$.
\end{cor}

\begin{proof}
    The statement follows from Theorem~\ref{thm:blue equals red} and Corollary~\ref{cor:singly even}, and the fact that each vertex is adjacent to exactly one vertex of the same color. 
\end{proof}

\subsection{Circulant graphs}\label{sub:circu} 

In this section, we provide colorings for circulant graphs, and we characterize cubic circulants. We begin with definitions and several general observations about circulants. Figures~\ref{fig:C12(1,5,6)} and \ref{fig:C14(1,6,7)} show the circulants $C_{12}(1,5,6)$ and $C_{14}(1,6,7)$.

\begin{defn} \rm Let $1\leq d_1<d_2< \cdots <d_s\leq \frac{n}{2}$. 
    The circulant graph $C_n(d_1,d_2,\dots,d_s)$ is the graph with vertex set $\{0,1,2,3, \ldots, n-1\}$ and $ij$ is an edge if $i - j \equiv \pm \; d_k \pmod n$  for some $k: 1 \le k \le s$. Given $S=\{d_1, d_2, \ldots, d_s\}$, we write $C_n(S)$ for $C_n(d_1,d_2,\dots,d_s)$, and we let $\overline{S} = \{1, 2, 3, \ldots, \frac{n}{2} \} - S$.
    \end{defn}

Circulants are regular and their degree is determined by the set $S$ of differences.

\begin{obs} \label{circ-deg} Let $S\subseteq \{1,2,3,4, \ldots, \frac{n}{2}\}$.
The circulant $C_n(S)$ is regular of degree $2|S|$ if $\frac{n}{2}\not\in S$, and regular of degree $2|S|-1$ if $\frac{n}{2}\in S$. 
\end{obs}

The complement of a circulant is also a circulant, which we record in the next observation. 

\begin{obs} \label{obs:comple} If $S\subseteq \{1,2,3,4, \ldots, \frac{n}{2}\}$, then $\overline{C_n(S})=C_n(\overline{S})$. 
\end{obs}

Note that the circulant $C_8(2)$ consists of the sum of two copies of the cycle $C_4$, and so a CNB-coloring of each copy of $C_4$ provides a CNB-coloring of $C_8(2)$.  We generalize this in the next proposition.  

\begin{prop} \label{gcd-div} 
    Let $t$ be the greatest common divisor of $d_1, d_2, \ldots, d_s, n$ for $1\leq d_1<d_2< \cdots <d_s\leq \frac{n}{2}$. Then $C_{n}(d_1, d_2, \ldots, d_s)$ is a CNBC graph (respectively, a NBC graph) if and only if  $C_{\frac{n}{t}}(\frac{d_1}{t}, \frac{d_2}{t}, \ldots, \frac{d_s}{t})$ is a CNBC graph (respectively, an NBC graph). 
\end{prop} 

\begin{proof}
The graph $C_{n}(d_1, d_2, \ldots, d_s)$  is isomorphic to $t$ disjoint copies of $C_{\frac{n}{t}}(\frac{d_1}{t}, \frac{d_2}{t}, \ldots, \frac{d_s}{t})$.  The result holds since a graph with more than one component is a CNBC graph (respectively, an NBC graph) if and only if each component is a CNBC graph (respectively, an NBC graph). 
\end{proof} 

\subsubsection{Circulants and complements} \label{subsub:circ-comp} 
In this section, we  construct infinite families of CNBC and NBC circulants, and construct an interesting family of quintic circulants in Example~\ref{exa:quint}.

\begin{prop} \label{circ-comp} 
The circulant $C_n(S)$ is an NBC graph if and only if $C_n(\overline{S})$ is a CNBC graph.
\end{prop}
\begin{proof}
Circulants are regular graphs, and therefore by Theorem~\ref{thm:blue equals red} and Theorem~\ref{thm:blue equals red-nb},  if $C_n(S)$ is an NBC graph or a CNBC graph, then in any $(R,B)$ coloring, $|R|=|B|$. The statement follows by Theorem~\ref{thm:complement} and Observation~\ref{obs:comple}. 
\end{proof}

We prove that some families of circulants are NBC graphs and some are CNBC graphs. 

\begin{thm} \label{thm:circulant-n-2} 
Let $n$ and $s$  be  even integers and  let $S=\{d_1, d_2, \ldots, d_s\}$ where $1< d_1< d_2< \cdots <d_s < \frac{n}{2} -1$.  If exactly half of $d_1, d_2, \ldots, d_s$ are even,  then the circulant $C_n(S)$ is an NBC graph. If in addition,   $ n \equiv 2 \pmod 4$, then $C_n(S\cup \{\frac{n}{2}\})$ is a CNBC graph. 

\end{thm}

\begin{proof}
Color the vertices of $C_n(S)$ alternately with red and blue, so vertex $i$ is red when $i$ is odd and blue when $i$ is even. For $1\leq i\leq s$, when $d_i$ is even, edges exist between numbered vertices of the same parity, and when $d_i$ is odd, edges exist between numbered vertices of the opposite parity. Thus, each vertex has an equal number of red and blue neighbors, and $C_n(S)$ is an NBC graph.

For the second statement, we note that $\frac{n}{2}$ is odd, and this difference causes each red vertex to have one additional blue neighbor, and each blue vertex to have one additional red neighbor. Thus, the closed neighborhood of each vertex in $C_n(S\cup\{ \frac{n}{2}\})$ is balanced.  Alternatively, $C_n(\overline{S\cup\{ \frac{n}{2}\}})$ is an NBC graph and use  Proposition~\ref{circ-comp}. \end{proof}

The graph $C_{14}(1,6,7)$ satisfies the hypothesis of Theorem~\ref{thm:circulant-n-2} where $S = \{1,6\}$, since $S$ contains one even and one odd integer and $ 14 \equiv 2 \pmod 4$.  A CNB-coloring is shown in Figure~\ref{fig:C14(1,6,7)}.

\begin{thm} \label{thm:circulant-n-0}
Let $ n \equiv 0 \pmod 4$ and $S = \{d_1, d_2, \ldots, d_s\}$ where $s$ is even and $1 \le d_1 < d_2 < \cdots < \frac{n}{2}$.  Furthermore, suppose that $d_i \in S$ if and only if $\frac{n}{2} - d_i \in S$.  Then the circulants $C_n(S)$ and $C(S\cup\{\frac{n}{4}\})$ are NBC graphs and the circulants 
  $C_n(S\cup\{\frac{n}{2}\})$ and  $C_n(S\cup\{\frac{n}{4}, \frac{n}{2}\})$ are CNBC graphs. \end{thm}

\begin{proof}
Color the even vertices in $\{0,1,2,,\ldots, \frac{n}{2}-1\} $ of $C_n(S)$ red and the odd ones blue. Color the even vertices in $\{ \frac{n}{2}, \frac{n}{2}+1, \ldots, n-1\} $  blue and the odd ones red. Since $\frac{n}{2}$ is even,  vertex $k$ and vertex $k + \frac{n}{2}  \pmod n$ 
have opposite colors for each $k: 0 \le k \le n-1$.  Consider any vertex $j \in C_n(S)$ and  its neighbors.  There are four neighbors  corresponding to $d_i$ and  $\frac{n}{2} - d_i$, namely,  $j+ d_i \pmod n$, $j- d_i \pmod n$, $j -d_i + \frac{n}{2} \pmod n$, and $j +d_i + 
\frac{n}{2} \pmod n$.  As noted above, $j+ d_i \pmod n$ and  $j +d_i + \frac{n}{2} \pmod n$ have opposite colors as do $j- d_i \pmod n$ and  $j -d_i + \frac{n}{2} \pmod n$.  Thus vertex $j$ has two red neighbors and two blue neighbors corresponding to $d_i$ and $\frac{n}{2} - d_i$ for each $i$ and therefore our coloring is an NBC coloring. Now $(j - \frac{n}{4} ) + \frac{n}{2} = j + \frac{n}{4} $ hence $j-\frac{n}{4}$ and $j+\frac{n}{4}$ are opposite colors. Thus, $C(S\cup\{\frac{n}{4}\})$ is an NBC graph.

For the graph $C_n(S\cup\{\frac{n}{2}\})$ and  $C_n(S\cup\{\frac{n}{4}, \frac{n}{2}\})$, we use the same coloring.  Since vertex $k$ and vertex $k + \frac{n}{2} \pmod n$ have opposite colors for each $k$, our coloring is a CNBC coloring. Alternatively, $C_n(\overline{S\cup\{ \frac{n}{2}\}})$ or $C_n(\overline{S\cup\{\frac{n}{4}, \frac{n}{2}\}})$ is an NBC graph and use  Proposition~\ref{circ-comp}.
\end{proof}

\begin{example} \label{exa:quint} \rm In this example we consider  $5$-regular circulant graphs $C_n(1,\frac{n}{2}-1,\frac{n}{2})$, and observe that they are CNBC graphs whenever $n$ is even. This follows from Theorem~\ref{thm:circulant-n-2} when $n \equiv 2 \pmod 4$,  because 1 is odd and $\frac{n}{2}-1$ is even, and from Theorem~\ref{thm:circulant-n-0} when $n \equiv 0 \pmod 4$, because $1+\frac{n}{2}-1=\frac{n}{2}$. Figure~\ref{fig:C12(1,5,6)} demonstrates the coloring for $C_{12}(1,5,6)$ and Figure~\ref{fig:C14(1,6,7)} demonstrates the coloring for $C_{14}(1,6,7)$. By the same theorems, or by Proposition~\ref{thm:complement}, $C_{12}(2,3,4)$ and $C_{14}(2,3,4,5)$ are NBC graphs.
\end{example}

\begin{thm} \label{thm:circulant-n-0-8}
Let $1 \le d_1 < d_2 < \cdots <  d_s< \frac{n}{2}$ and $S = \{d_1, d_2, \ldots, d_s\}$.  Further, let $s_1$ be the number of $d_i$ in $S$ that are congruent to $0 \pmod 4$ and $s_2$ be the number  that are congruent to $2 \pmod 4$.  Then the circulant   $C_n(S\cup\{\frac{n}{2}\})$ is a CNBC graph if either (i)  $ n \equiv 0 \pmod 8$ and $s_2 = s_1 + 1$ or (ii) $ n \equiv 4 \pmod 8$ and and $s_2 = s_1$.
 \end{thm}

\begin{proof}  
Color vertex $i$ red if  $i\equiv 0,1\pmod 4$ and blue if  $i\equiv 2,3 \pmod 4$ blue. First consider $d_j \equiv 1,3 \pmod 4$ and let $s_3$ be the number of  such $d_j$ in $S$.  For each  such $d_j$, we have $2d_j \equiv2 \pmod 4$, so 
vertices that differ by $2d_j$ will have opposite colors.  In particular, for any vertex $i$, the vertices $i+ d_j \pmod n$ and $i - d_j \pmod n$ have opposite colors.  So in $C_n(S\cup\{\frac{n}{2}\})$, each vertex $i$ has $s_3$ red neighbors and $s_3$ blue neighbors of the form $i \pm d_j \pmod n$ where $d_j \equiv 1,3 \pmod 4$.

It remains to consider neighbors of $i$ of the form $i \pm d_j \pmod n$ where $d_j = \frac{n}{2}$ or $d_j \equiv 0,2 \pmod 4$.   When $d_j \equiv 0 \pmod 4$, vertices $i \pm d_j \pmod n$    are the same color as $i$, and when $d_j \equiv 2 \pmod 4$, they   are the opposite color from $i$.  Thus in $C_n(S\cup\{\frac{n}{2}\})$, vertex $i$ has $2s_1$ neighbors of its color and $2s_2$ neighbors of the opposite color that have the form $i \pm d_j \pmod n$ where $d_j \equiv 0,2 \pmod 4$.  If $n \equiv 0 \pmod 8$ then vertex $i + \frac{n}{2}$ has the same color as $i$, so in $C_n(S\cup\{\frac{n}{2}\})$, the closed neighborhood $N[i]$ contains $s_3 + 2 + 2s_1$ vertices that are the same color as $i$ and $s_3 + 2s_2$ of the opposite color.  These quantities are equal when $s_2 = s_1 + 1$.
If $n \equiv 4 \pmod 8$ then vertex $i + \frac{n}{2}$ has the opposite color from $i$, so in $C_n(S\cup\{\frac{n}{2}\})$, the closed neighborhood $N[i]$ contains $s_3 + 2s_1 + 1$ vertices that are the same color as $i$ and $s_3 + 2s_2 + 1$ of the opposite color.  These quantities are equal when $s_2 = s_1$.
  \end{proof}

\begin{figure}[ht] 
\centering
\mbox{
    \begin{minipage}[b]{0.45\textwidth}
        \centering
        \begin{tikzpicture}[scale=.6]
            \def\n{12} 
            \def\R{3} 

            \foreach \i in {1,3,4,6,8,11} {
                \node[circle, draw, fill=red!30, minimum size=.4cm, inner sep=0pt] (O\i) at ({360/\n * \i}:\R) {\small R};
            }
            \foreach \i in {0,2,5,7,9,10} {
                \node[circle, draw, fill=blue!30, minimum size=.4cm, inner sep=0pt] (O\i) at ({360/\n * \i}:\R) {\small B};    
            }

            \foreach \i in {0, 1, ..., 11} {
                \pgfmathtruncatemacro{\nexti}{mod(\i+1,\n)}
                \draw (O\i) -- (O\nexti);
            }

            \foreach \i in {0, 1, ..., 11} {
                \pgfmathtruncatemacro{\nexti}{mod(\i+5,\n)}
                \draw (O\i) -- (O\nexti);
            }
    
            \foreach \i in {0, 1, ..., 11} {
                \pgfmathtruncatemacro{\nexti}{mod(\i+6,\n)}
                \draw (O\i) -- (O\nexti);
            }
        \end{tikzpicture}
        \caption{$C_{12}(1,5,6)$ with a CNB-coloring.}
        \label{fig:C12(1,5,6)}
    \end{minipage}
    \hspace{0.05\textwidth} 
    \centering
    \begin{minipage}[b]{0.45\textwidth}
        \centering
        \begin{tikzpicture}[scale=.6]
            \def\n{14} 
            \def\R{3} 

            \foreach \i in {0, 2, 4, 6, 8, 10, 12} {
                \node[circle, draw, fill=blue!30, minimum size=.4cm, inner sep=0pt] (O\i) at ({360/\n * \i}:\R) {\small B};
            }
            \foreach \i in {1, 3, 5, 7, 9, 11, 13} {
                \node[circle, draw, fill=red!30, minimum size=.4cm, inner sep=0pt] (O\i) at ({360/\n * \i}:\R) {\small R};    
            }

            \foreach \i in {0, 1, ..., 13} {
                \pgfmathtruncatemacro{\nexti}{mod(\i+1,\n)}
               \draw (O\i) -- (O\nexti);
            }

            \foreach \i in {0, 1, ..., 13} {
                \pgfmathtruncatemacro{\nexti}{mod(\i+6,\n)}
                \draw (O\i) -- (O\nexti);
            }
    
            \foreach \i in {0, 1, ..., 13} {
                \pgfmathtruncatemacro{\nexti}{mod(\i+7,\n)}
                \draw (O\i) -- (O\nexti);
            }
        \end{tikzpicture}
        \caption{$C_{14}(1,6,7)$ with a CNB-coloring.}
        \label{fig:C14(1,6,7)}
    \end{minipage}
}
\end{figure}

\subsubsection{Cubic and quintic circulants} 

In this section, we characterize cubic circulants and give a partial characterization of quintic circulants.  Note that all cubic circulants have the form $C_n(S)$ where $S=\{d, \frac{n}{2}\}$ and $1 \le d < \frac{n}{2}$ and all quintic circulants have the form $C_n(S)$ where $S=\{d_1, d_2 ,\frac{n}{2}\}$ and $1 \le d_1 < d_2 < \frac{n}{2}$.
Figures~\ref{fig:C12(1,6)} and \ref{fig:C12(5,6)} show examples of cubic circulants and Figures~\ref{fig:C12(1,5,6)} and \ref{fig:C14(1,6,7)} show examples of quintic circulants.

\begin{thm}\label{thm:cubic circulants} 
If $1\leq d\leq \frac{n}{2}-1$ and $d$ is relatively prime to $\frac{n}{2}$, then the cubic circulant $C_n(d,\frac{n}{2})$ is a CNBC graph if and only if $n \equiv 0\pmod 4$. 

\end{thm}

\begin{proof} If $C_n(d,\frac{n}{2})$ is a CNBC graph, then $n \equiv \ 0 \pmod  4$ by Corollary~\ref{3-reg-cor}.

Conversely, suppose $n \equiv \ 0 \pmod  4$. Since $d$ and $\frac{n}{2}$ are relatively prime, and $\frac{n}{2}$ is even, then $d$ must be odd. Color vertex $i$ red  if $i$ is odd and  blue if $i$ is even. Since $d$ is odd, each vertex $i$ has two neighbors of opposite parity ($i\pm d$) and one neighbor of the same parity $(i+\frac{n}{2})$.  Thus, the closed neighborhood of each vertex is balanced, and $C_n(d,\frac{n}{2})$ is a CNBC graph.
\end{proof}

Theorem~\ref{thm:cubic circulants} is still helpful for cubic circulants $C_n(d,\frac{n}{2})$ where $d$ and $\frac{n}{2}$ are not relatively prime, but we must first apply Proposition~\ref{gcd-div}. For example, the graph $C_{12}(4,6)$ is isomorphic to two copies of $C_6(2,3)$, and these are not CNBC graphs by Theorem~\ref{thm:cubic circulants}, so neither is $C_{12}(4,6)$.

\begin{figure}[ht]
\centering
\mbox{
    \begin{minipage}[b]{0.45\textwidth}
        \centering
        \begin{tikzpicture}[scale=.6]
            \def\n{12} 
            \def\R{3} 

            \foreach \i in {0, 2, 4, 6, 8, 10} {
                \node[circle, draw, fill=blue!30, minimum size=.4cm, inner sep=0pt] (O\i) at ({360/\n * \i}:\R) {\small B};
            }
            \foreach \i in {1, 3, 5, 7, 9, 11} {
                \node[circle, draw, fill=red!30, minimum size=.4cm, inner sep=0pt] (O\i) at ({360/\n * \i}:\R) {\small R};    
            }

            \foreach \i in {0, 1, ..., 11} {
                \pgfmathtruncatemacro{\nexti}{mod(\i+1,\n)}
                \draw (O\i) -- (O\nexti);
            }

            \foreach \i in {0, 1, ..., 11} {
                \pgfmathtruncatemacro{\nexti}{mod(\i+6,\n)}
                \draw (O\i) -- (O\nexti);
            }
        \end{tikzpicture}
        \caption{$C_{12}(1,6)$ with a CNB-coloring.}
        \label{fig:C12(1,6)}
    \end{minipage}
    \hspace{0.05\textwidth} 
    \centering
    \begin{minipage}[b]{0.45\textwidth}
        \centering
        \begin{tikzpicture}[scale=.6]
            \def\n{12} 
            \def\R{3} 

            \foreach \i in {0, 2, 4, 6, 8, 10} {
                \node[circle, draw, fill=blue!30, minimum size=.4cm, inner sep=0pt] (O\i) at ({360/\n * \i}:\R) {\small B};
            }
            \foreach \i in {1, 3, 5, 7, 9, 11} {
                \node[circle, draw, fill=red!30, minimum size=.4cm, inner sep=0pt] (O\i) at ({360/\n * \i}:\R) {\small R};    
            }

            \foreach \i in {0, 1, ..., 11} {
                \pgfmathtruncatemacro{\nexti}{mod(\i+5,\n)}
                \draw (O\i) -- (O\nexti);
            }

            \foreach \i in {0, 1, ..., 11} {
                \pgfmathtruncatemacro{\nexti}{mod(\i+6,\n)}
                \draw (O\i) -- (O\nexti);
            }
        \end{tikzpicture}
        \caption{$C_{12}(5,6)$ with a CNB-coloring.}
        \label{fig:C12(5,6)}
    \end{minipage}
}
\end{figure}

In the next theorem, we characterize quintic circulants with $n \equiv 2 \pmod 4$. We address the remaining cases of quintic circulants in Question~\ref{ques:circ} of Section~\ref{openq}. 

\begin{thm}\label{thm:circulant-n-2-quint} 
If $n \equiv 2 \pmod 4$ and $1\leq d_1<d_2<\frac{n}{2}$, then
the quintic circulant $C_n(d_1, d_2, \frac{n}{2})$ is a CNBC graph if and only if  $d_1\not \equiv d_2 \pmod 2$. 
\end{thm}

\begin{proof} 
Let $G = C_n(d_1, d_2, \frac{n}{2})$ and partition $V(G)$ as $X \cup Y$ where $X = \{0,2,4,6, \ldots, n\}$ and $Y = \{1,3,5, \ldots, n-1\}.$
Suppose that $n \equiv 2 \pmod 4$ and $1\leq d_1<d_2<\frac{n}{2}$.  
By Theorem~\ref{thm:circulant-n-2}, when $d_1$ and $d_2$ have opposite parity, the quintic circulant $C_n(d_1, d_2, \frac{n}{2})$ is a CNBC graph.  Conversely, suppose $d_1$ and $d_2$ have the same parity and for a contradiction, assume that $(R,B)$ is a CNB-coloring of $C_n(d_1, d_2, \frac{n}{2})$.
In any CNB-coloring of a quintic graph, every red vertex has two red neighbors and three blue neighbors. 

First suppose  that $d_1$ and $d_2$ are both odd.
Since $n \equiv 2 \pmod 4$, $\frac{n}{2}$ is also odd, and hence $C_n(d_1, d_2, \frac{n}{2})$ is bipartite with bipartition $X \cup Y$. Let $x_r$ be the number of red vertices in $X$ and $y_r$ be the number of red vertices in $Y$.  We can calculate  the number of edges in $RR$   by considering  their endpoint in $X$ to get $|RR| = 2x_r$ or  by considering  their endpoint in $Y$ to get $|RR| = 2y_r$.  Thus $x_r = y_r$ and $|R| = x_r + y_r = 2x_r$, which is even.  By Theorem~\ref{thm:blue equals red}, we have $|R| = \frac{n}{2}$, so $\frac{n}{2}$ is even.  This contradicts the hypothesis 
$n \equiv 2 \pmod 4$.

Now suppose $d_1$ and $d_2$ are both even.   Since $n \equiv 2 \pmod 4$, $C_n(d_1,d_2)=2C_{\frac{n}{2}}(\frac{d_1}{2}, \frac{d_2}{2})$,  with one copy induced by the vertices in $X$ and the other by the vertices in $Y$.  Since $\frac{n}{2}$ is odd, in $G$, each vertex in $X$ has exactly one neighbor in $Y$ in addition to its four neighbors in $X$.
Let $x_r$ and $y_r$ be the number of red vertices in $X$ and $Y$ respectively.  Similarly,  let  $x_b$ and $y_b$ be the number of blue vertices in $X$ and $Y$ respectively.   Note that $|X| = \frac{n}{2}$, which is odd, so we may assume without loss of generality that $x_r > x_b$. By Theorem~\ref{thm:blue equals red}, $|R| = |B|$ and we know $|R| + |B| = n$, so $|R| = \frac{n}{2}$.    Therefore, $x_r + y_r = |R| = \frac{n}{2} = |X| = x_r + x_b$ and we have $y_r = x_b$.  Similarly, $x_r = y_b$.

We consider the subgraph $G[X]$ induced by the vertices in $X$ and calculate in two ways the number of red endpoints of edges in $G[X]$.  On the one hand, this quantity is $4x_r$ since each red vertex in $X$ has degree $4$ in $G[X]$.  In this calculation, edges with one red and one blue endpoint are counted at their red endpoint.  On the other hand, we calculate this quantity by considering all vertices in $X$ and counting the number of red vertices they are adjacent to in $G[X]$.  In this way, edges with one red and one blue endpoint are counted at their blue endpoint.
Let $i$ be the number of blue vertices in $X$ whose neighbor in $Y$ is blue.  Such vertices have 3 red neighbors in $X$ (and one blue).  The remaining $x_b-i$ blue vertices in $X$ have a red neighbor in $Y$ and these vertices have two red neighbors in $X$ (and two blue).    Similarly, let $j$ be the number of red vertices in $X$ whose neighbor in $Y$ is blue.  Such vertices have 2 red neighbors in $X$ (and two blue).  The remaining $x_r-j$ red vertices in $X$ have a red neighbor in $Y$ and these vertices have one red neighbor in $X$ (and three blue).

Thus $4x_r = 3i + 2(x_b-i) + 2j + (x_r-j)$. Since the edges between $X$ and $Y$ in $G$ form a matching, we have $i + j = y_b$, so the equation above simplifies to
$3x_r = y_b + 2x_b$.  We showed above that $y_b = x_r$, thus we get $x_r = x_b$.  This contradicts $x_r > x_b$ from above.
\end{proof}

\subsection{Generalized Petersen graphs} \label{sub:peter}

In this section, we characterize those generalized Petersen graphs that are CNBC graphs.   Generalized Petersen graphs are introduced in \cite{W69}. 

\begin{defn} \rm Let $1\leq d\leq \frac{n-1}{2}$. 
The generalized Petersen graph, $GP(n,d)$ consists of a cycle $C_n$ with consecutive vertices 
$c_0,c_1, c_2, \ldots, c_{n-1}$ and additional vertices $\hc_0, \hc_1, \hc_2, \ldots, \hc_{n-1}$, with an edge  between $c_i$ and $\hc_{i}$ for $0\leq i\leq n-1$, and edges between $\hc_i$ and $\hc_{i+d}$ for $0\leq i\leq n-1$, with subscript arithmetic modulo $n$. \end{defn}

Notice that $GP(n,d)$ is 3-regular and $GP(5,2)$ is isomorphic to the Petersen graph. The graph $GP(n,d)$ can be described as the circulant $C_n(1)$ with vertex set 
$\{c_0,c_1,  \ldots, c_{n-1}\}$  together with the circulant $C_n(d)$ with vertex set 
$\{\hc_0,\hc_1, \hc_2, \ldots, \hc_{n-1}\}$ and an edge between $c_i$ and $\hc_i$ for $0\leq i\leq n-1$. 

\begin{prop} If    $GP(n,d)$ is a CNBC graph, then  $n$ is even. 
\label{GP-n-even}
\end{prop}

\begin{proof}
The graph $GP(n,d)$ is $3$-regular and contains $2n$ vertices.  Thus $2n \equiv 0 \pmod 4 $ by Corollary~\ref{3-reg-cor}.  It follows that $n$ is even.
\end{proof}

\begin{figure}[ht]
\centering
\begin{tikzpicture}[scale=.65]
    \def\n{8} 
    \def\k{3} 
    \def\R{3} 
    \def\r{1.5} 

    \foreach \i in {0, 2, 4, 6} {
        \node[circle, draw, fill=blue!30, minimum size=.4cm, inner sep=0pt] (O\i) at ({360/\n * \i}:\R) {\small B};
    }
    \foreach \i in {1, 3, 5, 7} {
        \node[circle, draw, fill=red!30, minimum size=.4cm, inner sep=0pt] (O\i) at ({360/\n * \i}:\R) {\small R};    
    }

    \foreach \i in {0, 2, 4, 6} {
        \node[circle, draw, fill=blue!30, minimum size=.4cm, inner sep=0pt] (I\i) at ({360/\n * \i}:\r) {\small B};
    }
     \foreach \i in {1, 3, 5, 7} {
        \node[circle, draw, fill=red!30, minimum size=.4cm, inner sep=0pt] (I\i) at ({360/\n * \i}:\r) {\small R};
    }

    \foreach \i in {0, 1, ..., 7} {
        \pgfmathtruncatemacro{\nexti}{mod(\i+1,\n)}
        \draw (O\i) -- (O\nexti);
    }

    \foreach \i in {0, 1, ..., 7} {
        \pgfmathtruncatemacro{\nexti}{mod(\i+\k,\n)}
        \draw (I\i) -- (I\nexti);
    }

    \foreach \i in {0, 1, ..., 7} {
        \draw (O\i) -- (I\i);
    }
\end{tikzpicture}
\caption{$GP(8,3)$ with a CNBC coloring.}
\label{fig:GP(8,3)}
\end{figure}
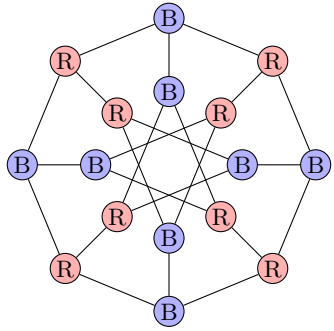

Figure \ref{fig:GP(8,3)} shows the CNBC graph $GP(8,3)$ with the coloring described in the proof of the next proposition. 

\begin{prop}\label{prop:GP(n,d)}
If $n$ is even and $d$ is odd, then $GP(n,d)$ is a CNBC graph.  
\end{prop}

\begin{proof} 
Alternate colors on $c_0,c_1, c_2, c_3, \ldots, c_{n-1}$, and color $\hc_i$ the same color as $c_i$, for $0\leq i\leq n-1$. Since $n$ is even, by parity, each vertex $c_i$ has two neighbors on the $n$-cycle $\{c_0,c_1, c_2, \ldots, c_{n-1}\}$ with the opposite color. Since $d$ is odd, $i-d, i+d$ have the opposite parity of $i$, and hence $c_{i-d}, c_{i+d}$ are colored the opposite color to $c_i$, and $\hc_{i-d}$ and $\hc_{i+d}$ are colored the opposite color to $\hc_i$. Thus, $N[c_i]$ and $N[\hc_i]$ are balanced. 
\end{proof}

\begin{thm}\label{thm:GP(n,d) not}
If $d$ is even, then $GP(n,d)$ is not a CNBC graph.  
\end{thm}

\begin{proof}
Suppose $d$ is even and assume for a contradiction that $(R,B)$ is a CNB-coloring of $GP(n,d)$. Let $C=\{c_0,c_1, c_2, \ldots, c_{n-1}\}$ and $\hat{C} = \{\hc_0,\hc_1, \hc_2, \ldots, \hc_{n-1}\}$.

We will do the proof by paying attention to the pairs $c_j, \hc_j$ that are the same color. We call this a monochromatic pair. Otherwise, a pair is bichromatic.  First, suppose that every such pair is monochromatic. Then $c_{j-1}$ and $c_{j+1}$ are the opposite color of $c_j$, which means that the colors on $C$ alternate. Since $d$ is even, $c_0, c_d, c_{-d}$ are all the same color, and $\hc_0, \hc_d, \hc_{-d}$ are all the same color, and these are the three vertices in $N[\hc_0]$, which is a contradiction. 

Second, suppose that every such pair is bichromatic. Then without loss of generality, if $c_0$ is red, $\hc_0$ is blue and one of $c_{-1}, c_1$ is red, the other blue. By symmetry, without loss of generality, suppose that $c_0$ and $c_1$ are red. Then $c_2$ is blue and $\hc_2$ is red, so $c_3$ is blue. Since $c_2$ and $c_3$ are both blue, then $c_4$ is red, and since $\hc_4$ is blue, then $c_5$ is 
red. This pattern continues, and by induction, $c_j$ is red if $j\equiv 0,1 \pmod 4$ and blue if $j\equiv 2,3 \pmod 4$. Since each pair is bichromatic,  $\hc_j$ is blue if $j\equiv 0,1 \pmod 4$ and red if $j\equiv 2,3 \pmod 4$. 
If $d \equiv  0 \pmod 4$, then $N[\hc_0] $ contains the three blue vertices $\hc_0, \hc_{-d}, \hc_d$, a contradiction, and if $d \equiv 2 \pmod 4$,  then $N[\hc_0]$ contains the three red vertices $c_0, \hc_{-d}, \hc_d$, also a contradiction. Thus, there must be a monochromatic pair. 

\begin{figure}[ht]
    \centering
    \begin{minipage}{0.45\textwidth}
        \centering
        \begin{tikzpicture}[scale=.6]
            \node[circle, draw, fill=blue!20,  minimum size=.4cm, inner sep=0pt, label=above:$c_{j-1}$] (cj-1) at (-1.5,1.5){B};
            \node[circle, draw, fill=red!20,  minimum size=.4cm, inner sep=0pt, label=above:$c_j$] (cj) at (0,1.5){R};
            \node[circle, draw, fill=blue!20,  minimum size=.4cm, inner sep=0pt, label=above:$c_{j+1}$] (cj1) at (1.5,1.5){B};
            \node[circle, draw, fill=red!20,  minimum size=.4cm, inner sep=0pt, label=right:$\hc_j$] (hcj) at (0,0) {R};
            \node[circle, draw, fill=blue!20,  minimum size=.4cm, inner sep=0pt, label=left:$\hc_{j-d}$] (hc{j-d}) at (-1.5,-1){B};
            \node[circle, draw, fill=blue!20,  minimum size=.4cm, inner sep=0pt, label=right:$\hc_{j+d}$] (hc{j+d}) at (1.5,-1){B};
            
            \draw (hcj) -- (cj);
            \draw (hcj) -- (hc{j-d});
            \draw (hcj) -- (hc{j+d});
            \draw (cj-1) -- (cj) -- (cj1);
        \end{tikzpicture}
        \caption{Observation (i).}
        \label{fig:obs1}
    \end{minipage}
    \hfill
    \begin{minipage}{0.45\textwidth}
        \centering
        \begin{tikzpicture}[scale=.6]
            \node[circle, draw, fill=red!20,  minimum size=.4cm, inner sep=0pt, label=above:$c_j$] (cj) at (0,1.5){R};
            \node[circle, draw, fill=red!20,  minimum size=.4cm, inner sep=0pt, label=below:$\hc_j$] (hcj) at (0,0) {R};

            \node[circle, draw, fill=blue!20,  minimum size=.4cm, inner sep=0pt, label=above:$c_{j+1}$] (cj1) at (1.5,1.5){B};
            \node[circle, draw, fill=blue!20,  minimum size=.4cm, inner sep=0pt, label=below:$\hc_{j+1}$] (hcj1) at (1.5,0) {B};

                \node[circle, draw, fill=red!20,  minimum size=.4cm, inner sep=0pt, label=above:$c_{j+2}$] (cj2) at (3,1.5){R};
            \node[circle, draw, fill=red!20,  minimum size=.4cm, inner sep=0pt, label=below:$\hc_{j+2}$] (hcj2) at (3,0) {R};

            \draw (hcj) -- (cj);
            \draw (hcj1) -- (cj1);
            \draw (hcj2) -- (cj2);
            \draw (cj)--(cj1)--(cj2);
        \end{tikzpicture}        
        \caption{Observation (ii).}
        \label{fig:obs2}
    \end{minipage}
\end{figure}

\begin{figure}[ht]
    \centering
    \begin{minipage}{0.45\textwidth}
        \centering
        \begin{tikzpicture}[scale=.6]
            \node[circle, draw, fill=red!20,  minimum size=.4cm, inner sep=0pt, label=above:$c_{j+2}$] (cj2) at (3,1.5){R};
            \node[circle, draw, fill=blue!20,  minimum size=.4cm, inner sep=0pt, label=below:$\hc_j$] (hcj) at (0,0) {B};
            
            \node[circle, draw, fill=red!20,  minimum size=.4cm, inner sep=0pt, label=above:$c_j$] (cj) at (0,1.5){R};
            \node[circle, draw, fill=blue!20,  minimum size=.4cm, inner sep=0pt, label=below:$\hc_{j+1}$] (hcj1) at (1.5,0) {B};
            \node[circle, draw, fill=blue!20,  minimum size=.4cm, inner sep=0pt, label=above:$c_{j+1}$] (cj1) at (1.5,1.5){B};
            \node[circle, draw, fill=red!20,  minimum size=.4cm, inner sep=0pt, label=below:$\hc_{j+2}$] (hcj2) at (3,0) {R};
            \node[circle, draw, fill=red!20,  minimum size=.4cm, inner sep=0pt, label=above:$c_{j+2}$] (cj2) at (3,1.5){R};

            \draw (hcj) -- (cj);
            \draw (hcj1) -- (cj1);
            \draw (hcj2) -- (cj2);
            \draw (cj)--(cj1)--(cj2);

            \node at (4, .75) {or}; 

            \node[circle, draw, fill=red!20,  minimum size=.4cm, inner sep=0pt, label=above:$c_j$] (cj) at (5,1.5){R};
            \node[circle, draw, fill=blue!20,  minimum size=.4cm, inner sep=0pt, label=below:$\hc_j$] (hcj) at (5,0) {B};
            
            \node[circle, draw, fill=red!20,  minimum size=.4cm, inner sep=0pt, label=above:$c_{j+1}$] (cj1) at (6.5,1.5){R};
            \node[circle, draw, fill=blue!20,  minimum size=.4cm, inner sep=0pt, label=below:$\hc_{j+1}$] (hcj1) at (6.5,0) {B};

            \node[circle, draw, fill=blue!20,  minimum size=.4cm, inner sep=0pt, label=above:$c_{j+2}$] (cj2) at (8,1.5){B}; \node[circle, draw, fill=red!20,  minimum size=.4cm, inner sep=0pt, label=below:$\hc_{j+2}$] (hcj2) at (8,0) {R};
            
            \draw (hcj) -- (cj);
            \draw (hcj1) -- (cj1);
            \draw (hcj2) -- (cj2);
            \draw (cj)--(cj1)--(cj2);
        \end{tikzpicture}        
        \caption{Observation (iii).}
        \label{fig:obs3}
    \end{minipage}
    \hfill
    \begin{minipage}{0.45\textwidth}
        \centering
        \begin{tikzpicture}[scale=.6]
            \node[circle, draw, fill=red!20,  minimum size=.4cm, inner sep=0pt, label=above:$c_j$] (cj) at (0,1.5){R};
            \node[circle, draw, fill=red!20,  minimum size=.4cm, inner sep=0pt, label=below:$\hc_j$] (hcj) at (0,0) {R};

            \node[circle, draw, fill=blue!20,  minimum size=.4cm, inner sep=0pt, label=above:$c_{j+1}$] (cj1) at (1.5,1.5){B};
            \node[circle, draw, fill=blue!20,  minimum size=.4cm, inner sep=0pt, label=below:$\hc_{j+1}$] (hcj1) at (1.5,0) {B};

            \node[circle, draw, fill=red!20,  minimum size=.4cm, inner sep=0pt, label=above:$c_{j+2}$] (cj2) at (3,1.5){R};  \node[circle, draw, fill=blue!20,  minimum size=.4cm, inner sep=0pt, label=below:$\hc_{j+2}$] (hcj2) at (3,0) {B};

            \draw (hcj) -- (cj);
            \draw (hcj1) -- (cj1);
            \draw (hcj2) -- (cj2);
            \draw (cj)--(cj1)--(cj2);

             \node at (4, .75) {or}; 

            \node[circle, draw, fill=blue!20,  minimum size=.4cm, inner sep=0pt, label=above:$c_j$] (cj) at (5,1.5){B};
            \node[circle, draw, fill=red!20,  minimum size=.4cm, inner sep=0pt, label=below:$\hc_j$] (hcj) at (5,0) {R};

            \node[circle, draw, fill=red!20,  minimum size=.4cm, inner sep=0pt, label=above:$c_{j+1}$] (cj1) at (6.5,1.5){R};
            \node[circle, draw, fill=blue!20,  minimum size=.4cm, inner sep=0pt, label=below:$\hc_{j+1}$] (hcj1) at (6.5,0) {B};

            \node[circle, draw, fill=red!20,  minimum size=.4cm, inner sep=0pt, label=above:$c_{j+2}$] (cj2) at (8,1.5){R};
            \node[circle, draw, fill=blue!20,  minimum size=.4cm, inner sep=0pt, label=below:$\hc_{j+2}$] (hcj2) at (8,0) {B};

            \draw (hcj) -- (cj);
            \draw (hcj1) -- (cj1);
            \draw (hcj2) -- (cj2);
            \draw (cj)--(cj1)--(cj2);
        \end{tikzpicture}        
        \caption{Observation (iv).}
        \label{fig:obs4}
    \end{minipage}
\end{figure}

\begin{figure}[ht]
    \centering
    \begin{minipage}{0.45\textwidth}
        \centering
    \begin{tikzpicture}[scale=.6]

        \node[circle, draw, fill=red!20,  minimum size=.4cm, inner sep=0pt, label=above:$c_{j-3}$] (cj-3) at (-4.5,1.5){R};

        \node[circle, draw, fill=blue!20,  minimum size=.4cm, inner sep=0pt, label=above:$c_{j-2}$] (cj-2) at (-3,1.5){B};
        \node[circle, draw, fill=red!20,  minimum size=.4cm, inner sep=0pt, label=below:$\hc_{j-2}$] (hcj-2) at (-3,0) {R};

        \node[circle, draw, fill=blue!20,  minimum size=.4cm, inner sep=0pt, label=above:$c_{j-1}$] (cj-1) at (-1.5,1.5){B};
        \node[circle, draw, fill=red!20,  minimum size=.4cm, inner sep=0pt, label=below:$\hc_{j-1}$] (hcj-1) at (-1.5,0) {R};

        \node[circle, draw, fill=red!20,  minimum size=.4cm, inner sep=0pt, label=above:$c_j$] (cj) at (0,1.5){R};
        \node[circle, draw, fill=red!20,  minimum size=.4cm, inner sep=0pt, label=below:$\hc_j$] (hcj) at (0,0) {R};

        \node[circle, draw, fill=blue!20,  minimum size=.4cm, inner sep=0pt, label=above:$c_{j+1}$] (cj1) at (1.5,1.5){B};
        \node[circle, draw, fill=red!20,  minimum size=.4cm, inner sep=0pt, label=below:$\hc_{j+1}$] (hcj1) at (1.5,0) {R};

        \node[circle, draw, fill=blue!20,  minimum size=.4cm, inner sep=0pt, label=above:$c_{j+2}$] (cj2) at (3,1.5){B};  \node[circle, draw, fill=red!20,  minimum size=.4cm, inner sep=0pt, label=below:$\hc_{j+2}$] (hcj2) at (3,0) {R};

        \node[circle, draw, fill=red!20,  minimum size=.4cm, inner sep=0pt, label=above:$c_{j+3}$] (cj3) at (4.5,1.5){R};  
            
            \draw (hcj) -- (cj);
            \draw (hcj1) -- (cj1);
            \draw (hcj2) -- (cj2);
            \draw (hcj-1) -- (cj-1);
            \draw (hcj-2) -- (cj-2);
            \draw (cj-3)--(cj-2)--(cj-1)--(cj)--(cj1)--(cj2)--(cj3);
    \end{tikzpicture}        
    \caption{Observation (v).}
    \label{fig:obs5}
    \end{minipage}
\end{figure}

In the rest of the proof, we consider cases based on the locations of monochromatic pairs. The following observations are pictured in Figures \ref{fig:obs1} through \ref{fig:obs5}. Note that the colors red and blue can be interchanged in each one to get a new statement. 

\begin{enumerate}[(i)]

    \item \label{a-both} Suppose $c_j$ and $\hc_j$ are both red. Then $c_{j-1}, c_{j+1}, \hc_{j-d}$ and $\hc_{j+d}$ are all blue. 
    
    \item \label{3-equa} Suppose that the colors of  $(\hc_j, \hc_{j+1}, \hc_{j+2})$ are (red, blue, red). We claim that the colors of $c_j, c_{j+1}, c_{j+2}$ are (red, blue, red) and the pairs $c_j,\hc_j$; $c_{j+1}, \hc_{j+1}$; $c_{j+2}, \hc_{j+2}$  are consecutive monochromatic pairs. The claim follows because if $c_{j+1}$ is red, then $c_j$ and $c_{j+2}$ each have two red neighbors, and are therefore both blue, but then $N[c_{j+1}]$ contains the three blue vertices, $c_j, c_{j+2}$, and $\hc_{j+1}$, which is a contradiction. 
    
    \item \label{bbred} Suppose that the colors of  $(\hc_j, \hc_{j+1}, \hc_{j+2})$ are (blue, blue, red).
    Then the colors of $c_j, c_{j+1}, c_{j+2}$ are (red, blue, red) or (red, red, blue), since $c_j$ cannot be blue. 
    \item \label{redbb} Suppose that the colors of  $(\hc_j, \hc_{j+1}, \hc_{j+2})$ are (red, blue, blue).
     Then the colors of $c_j, c_{j+1}, c_{j+2}$ are (red, blue, red) or (blue, red, red), since $c_{j+2}$ cannot be blue. 
  
    \item \label{equa-no-consec} Suppose that $c_j, \hc_j$ are red and $c_{j-1}, \hc_{j-1}$ and $c_{j+1}, \hc_{j+1}$ are both bichromatic pairs. Then $c_{j-2},c_{j-1}, c_{j+1}, c_{j+2}$ are blue and $c_{j-3}, \hc_{j-2}, \hc_{j-1}, \hc_{j+1}, \hc_{j+2}, c_{j+3}$ are red. 
\end{enumerate}

Suppose that  the longest consecutive sequence of monochromatic pairs has length at least $d+1$, and assume that the $d+1$ pairs $c_j, \hc_j$ for $0\leq j\leq d$ are monochromatic.  We observe that by~(\ref{a-both}), and using induction, the colors of $c_0, c_1, \ldots, c_d$ alternate, and since the pairs are monochromatic, the colors of $\hc_0, \hc_1, \ldots, \hc_d$ alternate also. Without loss of generality, suppose $c_0,\hc_0$ are red. Then $N[\hc_0]$ contains $c_0, \hc_0$ and $\hc_d$, and all three are red. This contradicts the assumption that the coloring is a CNB-coloring, and shows that there are no more than $d$ consecutive monochromatic pairs. 

Now, assume that the longest consecutive sequence of monochromatic pairs has length $t$, with $1\leq t\leq d$. If $t\geq 2$, then without loss of generality, assume that the $t$ monochromatic pairs are 
$c_j, \hc_j$ for $0\leq j\leq t-1$. We observe that by~(\ref{a-both}), the colors of $\hc_{j-d}, \hc_{j+d}$ are the same for $0\leq j\leq t-1$, and the colors of $\hc_{j-d}$ (and $\hc_{j+d}$)  alternate for $0\leq j\leq t-1$, so both sequences are 
consecutive sequences of monochromatic pairs of length $t$. Without loss of generality, let $c_{t-2}, \hc_{t-2}$ be blue and $c_{t-1}, \hc_{t-1}$ be red. Then $c_t$ must be blue and $\hc_t$ must be red. By~(\ref{a-both}), we have $\hc_{t-2-d}, \hc_{t-2+d}$ are red and $\hc_{t-1-d},\hc_{t-1+d}$ are blue. Since $\hc_{t}$ is 
red and is adjacent to the blue vertex $c_t$, one of $\hc_{t-d}, \hc_{t+d}$ is red and the other is blue. Then since the colors of one of $(\hc_{t-2-d},\hc_{t-1-d},\hc_{t-d})$ or $(\hc_{t-2+d},\hc_{t-1+d},\hc_{t+d})$ are (red, blue, red), then  
by~(\ref{3-equa}), there are three consecutive monochromatic pairs either for $t-2-d\leq j\leq t-d$ or $t-2+d\leq j\leq t+d$. This shows that there is a consecutive sequence of monochromatic pairs of length at least $t+1$, contradicting our choice of $t$. 
Thus, in any CNB-coloring, there can be no consecutive monochromatic pairs.

\begin{table}[ht]\centering
\begin{tabular}{c|ccccccccccc}
$j$ & $-5$ &  $-4$ & $-3$ & $-2$ & $-1$ & $0$ & $1$ & $2$ & $3$ & $4$  & $5$ \\ \hline 
$c_j$ & &  & red & blue & blue & {\color{red} red} & blue & blue & red &  & \\
$\hc_j$ & &  & & red & red & {\color{red} red} & red & red &&  &  \\ \hline
$c_{j-d}$ & &  & &  &  &  & & &  &   &   \\
$\hc_{j-d}$&  &   & &  &  & blue  & & & &  & \\ \hline 
$c_{j+d}$ &&  & &  &  &  & & &  &   &    \\
$\hc_{j+d}$&  &  & &  &  & blue & & &  &   & \\ \hline 
\end{tabular}\caption{Assume that $c_0, \hc_0$ are both red, then we fill in colors of other vertices using (\ref{equa-no-consec}).}\label{tab:1}
\end{table}

Without loss of generality, suppose that $c_0, \hc_0$ are red. By~(\ref{equa-no-consec}), we have $c_{-2}, c_{-1}, c_1, c_2$, $\hc_{-d}, \hc_d$ are blue, and $c_{-3}, \hc_{-2}, \hc_{-1},\hc_1, \hc_2,c_3$ are red. These are shown in the Table \ref{tab:1}.  

Since $\hc_1$ is red with a blue neighbor, one of $\hc_{1-d}, \hc_{1+d}$ is red 
and the other blue. By symmetry, and without loss of generality, let $\hc_{1-d}$ be red and $\hc_{1+d}$ be blue. By~\ref{3-equa}, if $(\hc_{-1-d}, \hc_{-d}, \hc_{1-d})=$ (red, blue, red), then $(c_{-1-d}, c_{-d}, c_{1-d})=$ (red, blue, red) and this makes 3 consecutive monochromatic pairs. Thus, $\hc_{-1-d}$ is blue. By~(\ref{bbred}), the colors of $(c_{-1-d}, c_{-d}, c_{1-d})$ are (red, blue, red) or (red, red, blue). However, the colors cannot be (red, blue, 
red), because that will yield two consecutive monochromatic pairs, so they must be (red, red, blue). Thus, $c_{-2-d}, c_{2-d}$ are blue, and $\hc_{2-d}, c_{3-d}$ are red. Since $\hc_{-1}$ is red, and its neighbors $c_{-1}, \hc_{-1-d}$ are blue, then $\hc_{-1+d}$ is red. Now the colors of $(\hc_{-1+d}, \hc_{d}, \hc_{1+d})$ are (red, blue, blue). By~(\ref{redbb}), the colors of  
$(c_{-1+d}, c_{d}, c_{1+d})$ are (red, blue, red) or (blue, red, red). They cannot be (red, blue, red), because that will yield two consecutive monochromatic pairs, so they must be (blue, red, red). Then $c_{1+d}$ is red and 
has a red neighbor, so $c_{2+d}$ is blue, and $c_{-1+d}$ is blue and has two red neighbors, so $c_{-2+d}$ is blue, and $\hc_{-2+d}, c_{-3+d}$ are red. Since $\hc_{-2}$ is red and has a red neighbor, $\hc_{-2+d}$, then $\hc_{-2-
d}$ is blue, and $c_{-2-d}, \hc_{-2-d}$ is a monochromatic pair, and we use (\ref{equa-no-consec}) to fill in the colors to the left. 
Since $\hc_2$ is red and has a red neighbor $\hc_{2-d}$, then $\hc_{2+d}$ is blue, and $c_{2+d}, \hc_{2+d}$ is a monochromatic pair, and we use (\ref{equa-no-consec}) to fill in the colors to the right. 
We record these colors in the Table \ref{tab:2}. 

\begin{table}[ht]\centering
\begin{tabular}{c|ccccccccccc}
$j$ & $-5$ &  $-4$ & $-3$ & $-2$ & $-1$ & $0$ & $1$ & $2$ & $3$ & $4$  & $5$ \\ \hline 
$c_j$ & &  & red & blue & blue & red & blue & blue & red &  & \\
$\hc_j$ & &  & & red & red & red & red & red &&  &  \\ \hline
$c_{j-d}$ & blue &  red & red & blue & red & red & blue & blue & red&  &  \\
$\hc_{j-d}$&  & blue  & blue & blue & blue & blue & red & red & &  & \\ \hline 
$c_{j+d}$ &&  &  red & blue & blue & red & red & blue & red & red & blue \\
$\hc_{j+d}$&  &  & & red & red & blue & blue &  blue & blue & blue  & \\ \hline 
\end{tabular}\caption{The colors that result by assuming that $\hc_{1-d}$ is red and $\hc_{1+d}$ is blue.}\label{tab:2}
\end{table}

Now $\hc_3$ is adjacent to a red vertex, $c_3$, and a blue vertex, $\hc_{3+d}$. Thus, either $\hc_3$ is red and $\hc_{3-d}$ is blue or the reverse. However, if $\hc_3$ is red and $\hc_{3-d}$ is blue, then $c_4$ is blue and $c_{4-d}$ is red and $\hc_{4-d}$ is blue. Since $\hc_{4+d}$ is blue, then $\hc_4$ has three blue neighbors: $c_4,\hc_{4-d}, \hc_{4+d}$. This is a contradiction, so $\hc_3$ must be blue. A similar proof shows that $\hc_{-3}$ is blue. Since we started with a generic monochromatic pair $c_0,\hc_0$, and proved that the pairs $c_3,\hc_3$ and $c_{-3}, \hc_{-3}$ are bichromatic, then we have proved that no two monochromatic pairs are separated by exactly two bichromatic pairs. 

Since $\hc_3$ is blue, then $c_4$ is red and $\hc_4, c_5$ are blue. Similarly, $c_{-4}$ is red and $\hc_{-4}, c_{-5}$ are blue. Then the colors of the neighbors of $\hc_{-3+d}$ and $\hc_{3-d}$ are determined, and both are red. Hence $c_{-3+d}, \hc_{-3+d}$ and $c_{3-d}, \hc_{3-d}$ are monochromatic pairs. We fill in these entries in the Table \ref{tab:3} using (\ref{equa-no-consec}). Now $c_{2+d}, \hc_{2+d}$ is a monochromatic pair, therefore $c_{5+d}, \hc_{5+d}$ is not monochromatic, so $\hc_{5+d}$ is red. This determines the color of $\hc_5$, and it is blue, and $c_5,\hc_5$ is a monochromatic pair. By similar reasoning, $\hc_{-5}$ is blue, and $c_{-5}, \hc_{-5}$ is a monochromatic pair. Since we started with a generic monochromatic pair $c_0,\hc_0$, and proved that the next monochromoatic pair in either direction occurs after four bichromatic pairs, it must be that every fifth consecutive pair is monochromatic, and the rest are bichromatic. These colors are also recorded in the Table \ref{tab:3}. 

\begin{table}[ht]\centering
\begin{tabular}{c|ccccccccccc}
$j$ & $-5$ &  $-4$ & $-3$ & $-2$ & $-1$ & $0$ & $1$ & $2$ & $3$ & $4$  & $5$ \\ \hline 
$c_j$ & blue & red & red & blue & blue & red & blue & blue & red & red & blue   \\
$\hc_j$ & blue & blue   & blue & red & red & red & red & red & blue &  blue &  blue \\ \hline
$c_{j-d}$ & blue &  red & red & {\color{blue} blue} & red & red & blue & blue & {\color{red} red} & blue  & blue \\
$\hc_{j-d}$& red & blue  & blue & {\color{blue}blue} & blue & blue & red & red & {\color{red} red}  & red & red   \\ \hline 
$c_{j+d}$ & blue & blue & {\color{red} red} & blue & blue & red & red &{\color{blue} blue} & red & red & blue  \\
$\hc_{j+d}$& red & red  &  {\color{red} red}  & red & red & blue & blue & {\color{blue} blue} & blue & blue &  red \\ \hline 
\end{tabular}\caption{The coloring with four bichromatic pairs in between two monochromatic pairs.}\label{tab:3}
\end{table}

Starting with any red monochromatic pair $c_j, \hc_j$, the coloring repeats every ten vertices, as shown in Table \ref{tab:4}. Since $d$ is even, $d\equiv 0,2,4,6,8 \pmod {10}$. Now a red vertex in $\hat{C}$ in a monochromatic pair must have two blue neighbors in $\hat{C}$. The only even values that are blue are 4 and 6, so both $d$ and  $-d$  must be congruent to 4 or 6$\pmod {10}$. However, $\hc_{j+1}$ is red and adjacent to the blue vertex $c_{j+1}$,  so $\hc_{j+1}$ has a red and a blue neighbor in $\hat{C}$, but  $1+4=5$ and $1+6=7$ and both of these vertices are blue. Thus, this coloring is not a CNB-coloring, and this contradiction completes the proof.  

\begin{table}[ht]\centering
\begin{tabular}{c|ccccccccccc}
$j$ &  $0$ & $1$ & $2$ & $3$ & $4$  & $5$ & $6$ & $7$ & $8$ & $9$\\ \hline 
$c_j$ &  red & blue & blue & red & red & blue & red & red & blue & blue    \\
$\hc_j$ &  red & red & red & blue &  blue & blue & blue & blue & red & red  \\ 
\end{tabular}\caption{Repeated coloring every ten vertices.}\label{tab:4}
\end{table}
\end{proof}

The next corollary combines the results of Proposition~\ref{GP-n-even}, Proposition~\ref{prop:GP(n,d)}, and Theorem~\ref{thm:GP(n,d) not} to  get a complete characterization of those generalized Petersen graphs that are CNBC graphs.

\begin{cor}
The generalized Petersen graph $GP(n,d)$ is a CNBC graph if and only if $n$ is even and $d$ is odd.
\end{cor}

\section{Cartesian product and strong product} \label{sub:cartesian-strong} 
In this section we study the Cartesian product and strong product of CNBC and NBC graphs.  

Recall that the \emph{Cartesian product} of graphs $G$ and $H$, denoted by $G \squarespace H$, is the graph with vertex set $V(G \squarespace H) = \{(g,h): g \in V(G), h \in V(H)\}$  where two vertices $(g,h)$ and $(g', h')$ are adjacent in $G \squarespace H$ if and only if (1) $g=g'$ and $hh' \in E(H)$ or (2) $h=h'$ and $gg' \in E(G)$. We refer to these as edges of the first and second types.  By definition, vertex $(g,h)$ in $G \squarespace H$ has degree $\deg_G(g) + \deg_H(h)$. 

In \cite{FM24}, the authors show that the Cartesian product of NBC graphs is an NBC graph. If $G$ and $H$ are CNBC  graphs then by Observation~\ref{obs:odddegree}, each vertex in $G$ and each vertex in $H$ has odd degree, so each vertex in $G \squarespace H$ has even degree and $G \squarespace H$  is not a CNBC graph.  Similarly, the Cartesian product of a CNBC graph with an NBC graph cannot be an NBC graph, but in the next theorem we show that it is a CNBC graph.

\begin{thm}
    If $G$ is a CNBC graph and $H$ is an NBC graph, then $G \squarespace H$ is a CNBC graph. 
    \label{cnbc-box-nbc}
\end{thm}
\begin{proof}
Let $G$ be a graph with a CNB-coloring $(R_G,B_G)$, and $H$ be a graph with NB-coloring $(R_H,B_H)$. Define $(R,B)$ on $V(G)\times V(H)$ such that the color of $(g,h)$ is blue if $g\in R_G$ and $h\in R_H$ or $g\in B_G$ and $h\in B_H$, and the color of $(g,h)$ is red if $g\in R_G$ and $h\in B_H$ or $g\in B_G$ and $h\in R_H$. We will prove that $(R,B)$ is a CNB-coloring for $G \squarespace H$.

Consider $(g,h) \in V(G \squarespace H)$. Without loss of generality, suppose $g\in B_G$.  Since $(R_H,B_H)$ is an NB-coloring of $H$, there exists $j$ so that $(g,h)$ has  exactly $j$ neighbors in $R$ of the form $(g, h')$ and $j$ neighbors in $B$ of the form $(g,h')$.  Its remaining neighbors have the form $(g',h)$. Since $(R_G,B_G)$ is a CNB-coloring of $G$ we know that $g$ has $k$ neighbors in $B_G$ and $k+1$  neighbors in $R_G$ for some $k$.  If $h\in B_H$, then $(g,h)\in B$ and it has $k$ neighbors in $B$ of the form $(g',h)$ and $k+1$ neighbors in $R$ of the form $(g',h)$.    If $h\in R_H$, then $(g,h)\in R$ and it has $k$ neighbors in $R$ of the form $(g',h)$ and $k+1$ neighbors in $B$ of the form $(g',h)$.  Thus $(g,h)$ has $j+k$ neighbors of its own color and $j+k+1$ neighbors of the opposite color, so  $(R,B)$ is a CNB-coloring of $G \squarespace H$.
\end{proof}

The Cartesian product of $K_2$ with a cycle is sometimes called a prism, and was defined in \cite{BDS72}. 
Let $Y_n= K_2 \squarespace C_n$. In the next example, we describe   CNB-colorings of $Y_n$, and in Proposition~\ref{prop:prisms} we show that these are the only possible CNB-colorings of $Y_n$.

\begin{example} \rm \label{exa:prisms}  
When $n$ is even, observe that $Y_n$ is isomorphic to $GP(n,1)$. Thus, one CNB-coloring of $Y_n$ is the coloring from the proof of Theorem~\ref{thm:GP(n,d) not} where the vertices of each $C_n$ are alternately red and blue, and each vertex in one copy of $C_n$ has the same color as its neighbor in the other copy of $C_n$.  
 
When $n \equiv 0 \pmod 4$,  we get a second CNB-coloring of $Y_n$.  The proof of Theorem~2 in \cite{FM24} provides a CNB-coloring of $C_n$ where edges that have two red endpoints alternate with edges that have two blue endpoints. We use this coloring on one copy of $C_n$ and the opposite coloring on the other copy of $C_n$, described in the proof of Theorem~\ref{cnbc-box-nbc}.
 This coloring is different from the first one since the coloring on the $C_n$'s is the opposite instead of the same. 
\end{example}

\begin{prop} \label{prop:prisms}  All of the CNB-colorings of $Y_n$ are given in Example~\ref{exa:prisms}. 
    
\end{prop}

\begin{proof} 

We will show that these are the only CNB-colorings of $Y_n$. 
Since the degree of each vertex in $Y_n$ is 3, by Corollary~\ref{3-reg-cor}, $|V(Y_n)|\equiv \ 0 \pmod  4$, and, since $|V(Y_n)|=2n$, then $n$ must be even. Also by Corollary~\ref{3-reg-cor}, $|R|=|B|=n$, $|RB|=2n$, and $|RR|=|BB|=\frac{n}{2}$, and the monochromatic edges form a matching of every vertex in $Y_n$. Suppose there is a edge with both endpoints colored red between two vertices $u_1$ and $u_2$ on different copies of $C_n$. Then the neighbors of $u_1$ and $u_2$ on the copies of $C_n$ are all blue. By a similar argument for an edge between the two copies of $C_n$ with both endpoints colored blue, their neighbors on the $C_n$'s are red. By induction, every edge between the $C_n$'s is monochromatic and the colors alternate on each $C_n$, as described in the first paragraph of Example~\ref{exa:prisms}. 

In any other coloring, there are no monochromatic edges between the copies of $C_n$. In a particular $C_n$, there are at most $\frac{n}{4}$ edges with both endpoints colored red, since each such edge alternates with an edge with both endpoints colored blue on $C_n$. Because there are $\frac{n}{2}$ edges with both endpoints colored red in total, there must be $\frac{n}{4}$ on each $C_n$. Thus, $n\equiv \ 0 \pmod  4$, and the coloring is the CNB-coloring described in the second paragraph of Example~\ref{exa:prisms}. 
\end{proof}

We next consider products of CNBC graphs. The next theorem is a positive result. 

\begin{thm} \label{cnbc-nbc-nbc}
 If $G$ is a CNBC graph, then $G \squarespace K_2$ is an NBC graph.
\end{thm}

\begin{proof}
Let $V(K_2) = \{h_0,h_1\}$, thus the vertices of $G \squarespace K_2$ have the form $(g,h_i)$ for $i \in \{0,1\}$ and $g \in V(G)$.  Let $(R,B)$ be a CNB-coloring of $G$.  We define  a coloring of $G \squarespace K_2$ as follows:  color $(g,h_i)$ red if $g \in R$ and blue if $g \in B$.  Consider vertex $(g,h_i)$.  There are an equal number of red and blue vertices in the set consisting of $(g,h_i)$ together with its neighbors of the form $(g',h_i)$ since $(R,B)$ is a CNB-coloring of $G$. In $G \squarespace K_2$ vertex $(g,h_i)$ has  one additional neighbor $(g,h_{1-i})$, which has the same color as $(g,h_i)$.  Thus, the open neighborhood of $(g,h_i)$ is balanced and our coloring is an NB-coloring of $G \squarespace K_2$.
\end{proof}

\begin{cor}
    For $n\geq 1$, the $n$-dimensional hypercube is an NBC graph if $n$ is even, and a CNBC graph if $n$ is odd. 
\end{cor}

\begin{proof}
Let $Q_n$ be the $n$-dimensional hypercube. Note that $Q_1=K_2$ and $K_2$ is a CNBC graph. Thus, by Theorem~\ref{cnbc-nbc-nbc}, $Q_1 \squarespace Q_1=Q_2$ is an NBC graph. By Theorem~\ref{cnbc-box-nbc}, $Q_2 \squarespace K_2=Q_3$ is a CNBC graph. The statement follows by induction on $n$. 
\end{proof}

In Theorem~\ref{cnbc-nbc-nbc}, we show that the Cartesian product of any CNBC graph $G$ and the CNBC graph $K_2$ is an NBC graph. However, the Cartesian product of two CNBC graphs need not be an NBC graph, as we show in the next example.

\begin{example} \rm In this example, we show that while 
  $K_4$ is a CNBC graph, the Cartesian product $K_4 \squarespace K_4 $ is not an NBC graph. 
  Let $G = K_4 \squarespace K_4 $ and label its vertices  by $x_{i,j}$ for $ 1 \le i,j \le 4$ so that for each $k: 1 \le k \le 4$, the vertices $x_{k,1}, x_{k,2}, x_{k,3}, x_{k,4} $ induce a $K_4$ and the vertices $x_{1,k},  x_{2,k},  x_{3,k},  x_{4,k}$ induce a $K_4$.  Suppose for a contradiction that $(R,B)$ is an NB-coloring of $G$. 
Each vertex of $G$ has degree $6$, so each has $3$ red neighbors and $3$ blue neighbors.  Consider the colors of the vertices in the $K_4$ induced by 
	$x_{1,1}, x_{1,2}, x_{1,3}, x_{1,4} $.

\smallskip
	\noindent
	Case 1:  All four of these vertices are the same color.  Without loss of generality  suppose they are all blue.  Each one has three blue neighbors, so the remaining
three neighbors must be red.  Thus $|B| = 4$ and $|R| = 12$, contradicting Theorem~\ref{thm:blue equals red-nb}.
 
	\smallskip
		\noindent
	Case 2:   Three of these vertices are the same color and one is the opposite color.  Without loss of generality we assume  $x_{1,1}, x_{1,2}, x_{1,3}$ are blue and $x_{1,4} $ is red.  Now $x_{1,4} $ has three blue neighbors, so its remaining neighbors $x_{2,4},  x_{3,4},  x_{4,4}$ are red.  Now each of  $x_{2,4},  x_{3,4},  x_{4,4}$ has three red neighbors, so their remaining three neighbors must be blue.  Thus  $|B| = 12$ and $|R| = 4$, again contradicting Theorem~\ref{thm:blue equals red-nb}.
	
	\smallskip
	
		\noindent
	Case 3:   Two of these vertices are red and two are blue.  Without loss of generality assume that $x_{1,1}$ and  $x_{1,2}$ are red and  $x_{1,3}$ and  $x_{1,4}$ are blue.  So $x_{1,1}$ has one red and two blue neighbors, so the remaining  neighbors must be one blue and two red.  Let $x_{k,1}$ be the blue neighbor.  So $x_{k,1}$ has three red neighbors of the form $x_{j,1}$, and thus its remaining three neighbors are blue, so $x_{k,1}$,  $x_{k,2}$, $x_{k,3}$,  and  $x_{k,4}$ are all blue.  Thus, this is the same contradiction as in Case 1.
\end{example}

Recall that the \emph{strong product} of graphs $G$ and $H$, denoted, $G \boxtimes H$, is the 
graph with vertex set $V(G \boxtimes  H) = \{(g,h): g \in V(G), h \in V(H)\}$ where two 
vertices $(g,h)$ and $(g',h')$ are adjacent in $G \boxtimes H$ if and only if 
(1) $g=g'$ and $hh' \in E(H)$ or 
(2) $h=h'$ and $gg' \in E(G)$ or 
(3) $gg'\in E(G)$ and $hh'\in E(H)$. We call these the first, second and third types of edges.

\begin{thm}\label{thm:strong}
    If $G$ is a CNBC graph, and $H$ be any graph, then $G \boxtimes H$ is a CNBC graph. 
\end{thm}

\begin{proof}
Let $W_h=\{(g,h)\;|\; g\in V(G)\}$. Then $W_h$ induces a copy of $G$ as a subgraph of $G\boxtimes H$. For each $h\in V(H)$, choose a CNB-coloring of $W_h$. We will show that this coloring is a CNB-coloring of $G\boxtimes H$. The neighbors of $(g,h)$ of the second type are all in $W_h$, and $W_h$ has a CNB-coloring, so $N_{W_h}[(g,h)]$ is balanced. For each $h'\in V(H)$ such that $hh'\in E(H)$, the neighbors of $(g,h)$ in $W_{h'}$ are exactly those in  $N_{W_{h'}}[(g',h')]$, of the first and third types. Since $W_{h'}$ is colored with a CNB-coloring, $N_{W_{h'}}[(g',h')]$ contains an equal number of red and blue vertices. This is true for every $h'$ such that $hh'\in E(H)$, and hence $N_{G\boxtimes H}[(g,h)]$ is balanced. 
\end{proof}

The hypothesis that $G$ is a CNBC graph is necessary, since if $G$ is a graph that is not a CNBC graph, such as $K_{1,3}$, then $G\boxtimes K_1$ is isomorphic to $G$ and is not a CNBC graph. Using Theorem~\ref{thm:strong} and $G=K_2$ gives a proof of the second part of Theorem~\ref{induced-subgraph}, that every graph $H$ is the induced subgraph of a CNBC graph. In the proof of Theorem~\ref{induced-subgraph}, we chose a particular coloring. In the proof of Theorem~\ref{thm:strong}, any coloring such that for each $h\in V(H)$, $(k_0,h)$ and $(k_1, h)$ are opposite colors is a CNB-coloring. 

\section{Conclusion}  \label{openq}

We conclude with some open questions, and acknowledgments. 

\begin{question} \label{ques:circ} \rm In Theorem~\ref{thm:circulant-n-2-quint}, we characterzied the quintic circulants whose number of vertices is congruent to $2 \pmod 4$, and in Theorems~\ref{thm:circulant-n-0} and \ref{thm:circulant-n-0-8}, we constructed infinite families of quintic CNBC circulants whose number of vertices is congruent to $0\pmod 4$. Assuming that not both $d_1$ and $d_2$ are even, the open cases are when a circulant $C_n(d_1, d_2, \frac{n}{2})$ has $n\equiv 0 \pmod 8$ vertices and $d_1\equiv 0 \pmod 4$ and $d_2$ is odd, or when $n\equiv 4 \pmod 8$  and either $d_1, d_2$ are odd and $d_1+d_2\neq \frac{n}{2}$, or one is odd and the other is even. 
Can the quintic CNBC circulants be completely characterized? 
\end{question}

\begin{question} \rm  
Recall that the \emph{direct product} of graphs $G$ and $H$, denoted, $G\times H$, is the 
graph with vertex set $V(G \times  H) = \{(g,h): g \in V(G), h \in V(H)\}$ where two 
vertices $(g,h)$ and $(g',h')$ are adjacent in $G\times H$ if and only if 
$gg'\in E(G)$ and $hh'\in E(H)$. 
 What can be said about $G\times H$ when $G$ or $H$ is a CNBC graph?  
\end{question}

\begin{question} \rm Is the problem of determining whether a graph is a CNBC graph an NP-complete problem? The analogous question for NBC graph is discussed in \cite{A24}. 
    
\end{question}

We gratefully acknowledge the American Institute of Mathematics (AIM) and their support of research collaboration. We started our project at the AIM workshop, {\it Graph Theory: structural properties, labelings, and connections to applications}, held July 22-26, 2024. We thank the workshop organizers, and we thank Novi Bong and Rinovia Simanjuntak for their participation in initiating the project.

\end{document}